\documentclass[12pt]{article}
\usepackage{latexsym,amsfonts,amssymb}
\usepackage{color}
\usepackage[toc,page]{appendix}
\usepackage{amsmath}
\usepackage{graphicx}
\usepackage{pdflscape}
\usepackage{rotating}
\usepackage{breqn}
\newtheorem{lemma}{Lemma}[section]
\allowdisplaybreaks

\newtheorem{thm}{Theorem}[section]

\newtheorem{rem}[thm]{Remark}

\date{}
\begin{document}

\title{\bf The infimum values of three probability functions for the Laplace distribution and the student's $t$ distribution}
\author{Rong-Sheng Hu, Ze-Chun Hu\thanks{Corresponding author: zchu@scu.edu.cn}, Zhen Huang, Mu-Xuan Li\\ \\
   {\small College of Mathematics, Sichuan University, Chengdu 610065, China}}

\maketitle

\begin{abstract}
\noindent
Let $\{X_\alpha\}$ be a family of random variables satisfying some distribution with a parameter $\alpha$, $E(X_{\alpha})$ be the expectation, and $Var(X_{\alpha})$ be the variance. In this paper, we study the infimum values of three probability functions:
$P(X_{\alpha}\leq y E(X_{\alpha}))$, $P\left(|X_{\alpha}-E(X_{\alpha})|\leq y\sqrt{Var(X_{\alpha})}\right)$ and $P\left(|X_{\alpha}-E(X_{\alpha})|\geq y\sqrt{Var(X_{\alpha})}\right), \forall y>0$, with respect to the parameter $\alpha$ for the Laplace distribution and the student's $t$ distribution. Our motivation comes from three former conjectures: Chv\'{a}tal's conjecture, Tomaszewski's conjecture and Hitczenko-Kwapie\'{n}'s conjecture.
\end{abstract}



\section{Introduction and main results}

In this paper, we  study the infimum values of three probability functions for the Laplace distribution and the student's $t$ distribution. To give the motivation, we recall three former conjectures at first.

Let $B(n,p)$ denote a binomial random variable with parameters $n$ and $p$. Let $\mathcal{R}$ be the family of random variables of the form $X=\sum^n_{k=1}a_k\varepsilon_k$, where $n\ge 1$,  $a_k, k=1, \dots, n,$ are real numbers with $\sum^n_{k=1} a_k^2=1$, and $\varepsilon_k$, $k=1, 2, \dots$, are independent
Rademacher random variables (i.e., $P(\varepsilon_k=1)=P(\varepsilon_k=-1)=1/2$).

Now we  state  three former conjectures as follows:

{\bf Chv\'{a}tal's Conjecture:} For any fixed $n\geq 2$, as $m$ ranges over $\{0,\ldots,n\}$, the probability $P(B(n,m/n)\leq m)$ is the smallest when
 $m$ is closest to $\frac{2n}{3}$.

{\bf Tomaszewski's conjecture:} $\inf_{X\in \mathcal{R}}P(|X|\leq 1)=\frac{1}{2}$.

{\bf Hitczenko-Kwapie\'{n}'s conjecture:} $\inf_{X\in \mathcal{R}}P(|X|\geq 1)=\frac{7}{32}$.

Chv\'{a}tal's Conjecture has applications in machine learning. Janson \cite{Jan21} showed that  Chv\'{a}tal's Conjecture holds for large $n$. Barabesi et al. \cite{BPR23} and Sun \cite{Sun21} proved that  Chv\'{a}tal's Conjecture is true  for general $n\geq 2$.

 Tomaszewski's conjecture has many applications in probability theory, geometric analysis and computer science.  Keller and Klein \cite{KK22} completely solved Tomaszewski's conjecture.

Hitczenko-Kwapie\'{n}'s conjecture was recently proved by Hollom and Portier \cite{HP23}. In fact, Hollom and Portier \cite{HP23} found an explicit expression for the anti-concentration function $y\mapsto \inf_{X\in\mathcal{R}} P\left(|X|\geq y\right), \forall y>0.$

Let $\{X_\alpha\}$ be a family of random variables satisfying some distribution with a parameter $\alpha$, $E(X_{\alpha})$ be the expectation, and $Var(X_{\alpha})$ be the variance.  For $y>0$, define three  functions as follows:
\begin{align*}
&C(y):=\inf_{\alpha}P(X_{\alpha}\leq yE(X_{\alpha})),\\
&T(y):=\inf_{\alpha}P\left(|X_{\alpha}-E(X_{\alpha})|\leq y\sqrt{Var(X_{\alpha})}\right),\\
&H(y):=\inf_{\alpha}P\left(|X_{\alpha}-E(X_{\alpha})|\geq y\sqrt{Var(X_{\alpha})}\right).
\end{align*}

Motivated by the above three former conjectures, the three functions $C(y), T(y)$ and $H(y)$ have been studied in a series of papers recently. In the following, we recall the existing results.

\begin{itemize}
\item Li et al. \cite{LXH23} studied $C(1)$ for the {\it Poisson, geometric}  and  {\it Pascal distributions} (as to the Pascal distribution, only partial results were obtained).

\item Guo et al. \cite{GTH24} studied $C(1)$ for the {\it negative binomial distribution} and gave an affirmative answer to the conjecture posed in \cite{LXH23} on the Pascal distribution.

\item Li et al. \cite{LHZ24} studied $C(y),\forall y>0,$ for the {\it Weibull distribution} and the {\it Pareto distribution}.

\item Hu et al. \cite{HLZZ24} studied $C(y),\forall y>0,$ for some infinitely divisible distributions including the {\it inverse Gaussian, log-normal, Gumbel} and {\it Logistic distributions}.

\item Zhou et al. \cite{ZLH24} studied $C(y),\forall y\in (0,1],$ for the {\it $F$ distribution}.

\item Sun et al. \cite{SHS24a} studied $C(y), \forall y>0,$ and $T(1)$ for the {\it Gamma distribution}.

\item Sun et al. \cite{SHS24b} studied $T(1)$ for some infinitely divisible continuous distributions including the {\it Lapace, Gumbel, Logistic, Pareto, infinitely divisible Weibull, log-normal, student's $t$} and {\it inverse Gaussian distributions}.

\item Sun et al. \cite{SHS23} studied $T(1)$ for {\it $F$ distribution}.

\item Zhang et al. \cite{ZHS25} studied $T(1)$ for {\it geometric, symmetric geometric, Poisson}  and {\it symmetric Poisson distributions} among other things.

\item Hu et al. \cite{HST24} studied $H(y), \forall y>0,$ for some familiar families of distributions, and the results show that   for certain familiar families
of distributions, including the {\it uniform, exponential, nondegenerate
Gaussian} and {\it student's $t$ distributions},  $H(y)$ is not identically zero; while for some other familiar families of distributions, including the {\it binomial, Poisson, negative binomial, hypergeometric, Gamma, Pareto, Weibull, log-normal} and {\it Beta distributions},  $H(y)$ is identically zero.

\item Xie \cite{Xie25} gave $C(y), T(y)$ and $H(y),\forall y>0$, for the the {\it Rayleigh distribution} and the {\it symmetric Rayleigh distribution}.

\item Zhao \cite{Zhao25} gave $H(y),\forall y>0$, for the {\it uniform, exponential, nondegenerate Gaussian} and {\it log-normal distributions} among other things.

\end{itemize}

Based on the existing results, up to now, we have known the three functions $C(y), T(y)$ and $H(y)$ completely for all $y>0$ only for the {\it uniform, exponential, nondegenerate Gaussian (normal), log-normal, Rayleigh distribution} and the {\it symmetric Rayleigh distribution}.

{\it The motivation of this paper is to  study the three functions $C(y), T(y)$ and $H(y)$ for the Laplace distribution and the student's $t$ distribution for all $y>0$.}

Let $X_{\mu,b}$ be a Laplace random variable with the location parameter $\mu$ and the scale parameter $b$, where $\mu\in \mathbb{R}$ and $b>0$. Then we know that the density function
$$
f_{\mu,b}(x)=\frac{1}{2b}e^{-\frac{|x-\mu|}{b}},
$$
and $E(X_{\mu,b})=\mu, Var(X_{\mu,b})=2b^2$. Now, for any $y>0$, we have
\begin{align*}
&C(y)=\inf_{\mu,b}P(X_{\mu,b}\leq y\mu),\\
&T(y)=\inf_{\mu,b}P\left(|X_{\mu,b}-\mu|\leq \sqrt{2}by\right),\\
&H(y)=\inf_{\mu,b}P\left(|X_{\mu,b}-\mu|\geq \sqrt{2}by\right).
\end{align*}

\begin{thm}\label{main-thm1}
Let $X_{\mu,b}$ be a Laplace random variable with the location parameter $\mu$ and the scale parameter $b$. Then \\
(i) $C(y)=0$ if $y>0$ and $y\neq 1$, and $C(1)=\frac{1}{2}$;\\
(ii) $T(y)=1-e^{-\sqrt{2}y},\forall y>0$;\\
(iii) $H(y)=e^{-\sqrt{2}y},\forall y>0$.
\end{thm}

\begin{rem}
 Theorem \ref{main-thm1}(ii) and (iii) show that the Laplace distribution possesses both  concentration property and  anti-concentration property.
\end{rem}

Let $X_v$ be a $t$-random variable with $v$ degrees of freedom, where $3 \le v \in \mathbb{N}$. Then we know that the density function
$$
f_{v}(x)
=\frac{\Gamma(\frac{v+1}{2})}{{\sqrt{v\pi} }\Gamma(\frac{v}{2}) }{(1+\frac{x^2}{v})^{-\frac{v+1}{2}}},
$$
and $E(X_{v})=0, Var(X_{v})=\sqrt{ \frac{v}{v-2} }$. Now, for any $y>0$, we have
\begin{align*}
&C(y)=\inf_{v \ge 3}P(X_{v}\leq 0),\\
&T(y)=\inf_{v \ge 3}P\left(|X_{v}|\leq y\sqrt{ \frac{v}{v-2} } \right),\\
&H(y)=\inf_{v \ge 3}P\left(|X_{v}|\geq y\sqrt{ \frac{v}{v-2} } \right).
\end{align*}
\indent For given $y\in (1,\sqrt{3})\cup (\sqrt{3},\infty)$, define
\begin{align}\label{1.1}
\left\{
\begin{array}{l}
C_1:= 2 y^6 +  \frac{3}{4}y^4 , \\
C_2:= \frac{9}{4}C_1 + \frac{25}{32}y^8 + \frac{3}{2} \left( C_1^{\frac{2}{3}} \cdot y^{\frac{8}{3}} + C_1^{\frac{1}{3}} \cdot y^{\frac{16}{3}} \right), \\
C_3:= \frac{13}{3}y^4 + 18y^2 + 2C_2 + 11, \\
V_1:= \max \left\{100, 8y^2 \right\} , \\
V_2:= \max \left\{ V_1, \frac{3y^4}{2}, \sqrt{2C_1} \right\}, \\
V_3:= \max \left\{ V_2, 2+\frac{2y^2}{1+2y}, 2+\frac{2y^4}{1+2y^2} \right\},\\
v_0(y):=  \max\!\left\{V_3, \ \frac{2C_{3}}{|y^2-3|} +1\right\}.
\end{array}
\right.
\end{align}
Define
$
\bar{v}_0(\sqrt{3}):= 1318.4.
$
Denote by $\Phi(x)$ the distribution function of the standard normal distribution (i.e. $\Phi(x)=\frac{1}{\sqrt{2\pi}}\int_{-\infty}^x e^{-\frac{u^2}{2}}du$), and  by  $F_{v}(x)$  the distribution function of $X_{v}$.

\begin{thm}\label{main-thm2}
Let $X_v$ be a t-random variable with $v$ degrees of freedom. Then \\
(i) $C(y)=\frac{1}{2},  \forall y>0$;\\
(ii) $T(y) =
\begin{cases}
2\Phi(y)-1, &\hfill 0 < y \le 1, \\

\min \left\{ {\min_{3 \le v \le [v_0(y)]+3} \left\{ 2F_{v} \left( y\sqrt{\frac{v}{v-2}} \right) -1 \right\} }, 2\Phi(y)-1 \right\}, &\hfill 1 < y < \sqrt3, \\

\min \left\{ \min_{3 \le v \le [\bar{v}_0(\sqrt{3})]+3} { \left\{ 2F_{v} \left( \sqrt{\frac{3v}{v-2}} \right) -1 \right\} }, 2\Phi \left( \sqrt{3} \right) - 1 \right\},  &\hfill y = \sqrt3, \\

\min_{3 \le v \le [v_0(y)]+3} \left\{ 2F_{v} \left( y\sqrt{\frac{v}{v-2}} \right) - 1 \right\},   & \hfill y > \sqrt3;
\end{cases}$\\
(iii) $H(y)=
\begin{cases}
 \min_{v=3,4}{ \left\{ 2 - 2F{_v} \left( y\sqrt{\frac{v}{v-2}} \right) \right\} }, &\hfill 0 < y \le 1, \\

\min_{3 \le v \le [v_0(y)]+3} \left\{ 2 - 2F_{v} \left( y\sqrt{\frac{v}{v-2}} \right) \right\},  &\hfill  1 < y < \sqrt3 ,  \\

\min_{3 \le v \le [\bar{v}_0(\sqrt{3})]+3} \left\{ 2 - 2F_{v}\left( \sqrt{\frac{3v}{v-2}} \right) \right\},  &\hfill y = \sqrt3, \\

\min{ \left\{ \min_{3 \le v \le [v_0(y)]+3}  \left\{ 2 - 2F_{v} \left( y\sqrt{\frac{v}{v-2}} \right) \right\} , 2 - 2\Phi(y)  \right\} },  &\hfill y > \sqrt3.
\end{cases}$
\end{thm}

\begin{rem}

In virtue of three-standard-deviation rule for the standard norm distribution, in this remark, we consider Theorem \ref{main-thm2}(ii)(iii) for $y=1,2,3$.

(i) For $ y = 1 $, we have
\begin{align*}
T(1)
&= 2\Phi(1)-1 \approx 0.6826, \\
H(1)&=\min_{v=3,4}\left\{2-2F_{v}\left(\sqrt{\frac{v}{v-2}}\right)\right\}=2-2F_{3} \left( \sqrt{3} \right)\approx 0.1817.
\end{align*}
They are the same with that in \cite{HST24} and \cite{SHS24b}.

(ii) For $ y = 2 $, we have
\begin{align*}
&C_0 = 3, C_1 = 140, C_2 \approx 1085.82, C_3 \approx 2323.97, \\
&V_1 = V_2 = V_3 = 100,\quad v_0(2) \approx 4648.94,
\end{align*}
and thus
\begin{align*}
T(2) &= \min_{3 \le v \le 4651} \left\{ 2F_{v} \left( 2\sqrt{\frac{v}{v-2}} \right) -1 \right\}, \\
H(2) &= \min \left\{ {\min_{3 \le v \le 4651} \left\{ 2 - 2F_{v} \left( 2\sqrt{\frac{v}{v-2}} \right) \right\} }, 2 - 2\Phi(2) \right\}.
\end{align*}
\indent By a Python program, we get that
$$
T(2) = \min_{3\le v\le 4651}\left\{ 2F_{v} \left( 2\sqrt{\frac{v}{v-2}}\right) -1 \right\}=
2F_{7} \left( 2\sqrt{\frac{7}{5}} \right) -1\approx 0.9501,
$$

Also, by a Python program, we get that
$$
\min_{3 \le v \le 4651} \left\{ 2 - 2F_{v} \left( 2\sqrt{\frac{v}{v-2}} \right) \right\}
= 2 - 2F_{3} \left( 2\sqrt{3} \right) \approx 0.0405,
$$
which together with $2-2\Phi(2)\approx 0.0455$ implies that
$$
H(2)=2 - 2F_{3} \left( 2\sqrt{3} \right) \approx 0.0405.
$$

(iii) For $y=3$, we have
\begin{align*}
&C_0 = 3, C_1 = 1518.75, C_2 \approx 18295.97, C_3 \approx 37115.94, \\
&V_1 = 100, V_2 = V_3 = 121.5, \quad v_0(3) \approx 12372.98,
\end{align*}
and thus
\begin{align*}
T(3) &= \min_{3 \le v \le 12375} \left\{ 2F_{v} \left( 3\sqrt{\frac{v}{v-2}} \right) - 1 \right\}, \\
H(3) &= \min{ \left\{ \min_{3 \le v \le 12375}  \left\{ 2 - 2F_{v} \left( 3\sqrt{\frac{v}{v-2}} \right) \right\} , 2 - 2\Phi(3)  \right\} }.
\end{align*}
\indent By a Python program, we get that
$$
T(3) = \min_{3\le v\le 12375}\left\{ 2F_{v} \left( 3\sqrt{\frac{v}{v-2}}\right) -1 \right\}=
2F_{3} \left( 3\sqrt{3} \right) -1\approx 0.986153,
$$

Also, by a Python program, we get that
$$
\min_{3 \le v \le 12375} \left\{ 2 - 2F_{v} \left( 3\sqrt{\frac{v}{v-2}} \right) \right\}
= 2 - 2F_{12375} \left( 3\sqrt{\frac{12375}{12375-2}} \right) \approx 0.002703,
$$
which together with $2-2\Phi(3)\approx 0.002700$ implies that
$$
H(3)=2-2\Phi(3) \approx 0.002700.
$$

\end{rem}

\begin{rem}
 Theorem \ref{main-thm2}(ii) and (iii) show that the student's $t$ distribution possesses both  concentration property and  anti-concentration property.
\end{rem}

The rest of this paper is organized as follows. The proofs of Theorem \ref{main-thm1} and Theorem \ref{main-thm2}  will be given in Section 2 and Section 3, respectively.

\section{Proof of Theorem \ref{main-thm1}}\label{S2}\setcounter{equation}{0}

The distribution function of $X_{\mu,b}$ is (see e.g. Appendix (i) of \cite{SHS24b})
\begin{align*}
F_{X_{\mu,b}}(x)=\left\{
\begin{array}{cl}
\frac{1}{2}e^{\frac{x-\mu}{b}},& \mbox{if}\ x<\mu,\\
1-\frac{1}{2}e^{-\frac{x-\mu}{b}},& \mbox{if}\ x\geq \mu.
\end{array}
\right.
\end{align*}
\indent (i) We have the following five cases:\\
\indent (i.1) $\mu>0,y\geq 1$. Now $y\mu\geq \mu$ and thus
\begin{align*}
C(y)&=\inf_{\mu,b}P(X_{\mu,b}\leq y\mu)=\inf_{\mu,b}\left(1-\frac{1}{2}e^{-\frac{(y-1)\mu}{b}}\right)
=\lim_{\frac{\mu}{b}\to 0+}\left(1-\frac{1}{2}e^{-\frac{(y-1)\mu}{b}}\right)=\frac{1}{2}.
\end{align*}
\indent (i.2) $\mu>0,0<y<1$. Now $y\mu<\mu$ and thus
\begin{align*}
C(y)&=\inf_{\mu,b}P(X_{\mu,b}\leq y\mu)=\inf_{\mu,b}\frac{1}{2}e^{-\frac{(1-y)\mu}{b}}
=\lim_{\frac{\mu}{b}\to+\infty}\frac{1}{2}e^{-\frac{(1-y)\mu}{b}}=0.
\end{align*}
\indent (i.3) $\mu=0$. Now $y\mu=0=\mu$ and thus
\begin{align*}
P(X_{\mu,b}\leq y\mu)=\frac{1}{2}e^{-\frac{(1-y)\mu}{b}}
\equiv \frac{1}{2}.
\end{align*}
Hence in this case, $C(y)=\frac{1}{2}$.\\
\indent (i.4) $\mu<0,y>1$. Now $y\mu<\mu$ and thus
\begin{align*}
C(y)&=\inf_{\mu,b}P(X_{\mu,b}\leq y\mu)=\inf_{\mu,b}\frac{1}{2}e^{\frac{(y-1)\mu}{b}}
=\lim_{\frac{\mu}{b}\to-\infty}\frac{1}{2}e^{\frac{(y-1)\mu}{b}}=0.
\end{align*}
\indent (i.5) $\mu<0,0<y\leq 1$. Now $y\mu\geq \mu$ and thus
\begin{align*}
C(y)&=\inf_{\mu,b}P(X_{\mu,b}\leq y\mu)=\inf_{\mu,b}\left(1-\frac{1}{2}e^{\frac{(1-y)\mu}{b}}\right)
=\lim_{\frac{\mu}{b}\to 0-}\left(1-\frac{1}{2}e^{\frac{(1-y)\mu}{b}}\right)=\frac{1}{2}.
\end{align*}
By (i.1)-(i.5), we obtain that (i) holds.

(ii) Since $b>0$ and $y>0$, we have
\begin{align*}
T(y)&=\inf_{\mu,b}P\left(|X_{\mu,b}-\mu|\leq \sqrt{2}by\right)\\
&=\inf_{\mu,b}\left[F_{X_{\mu,b}}(\mu+\sqrt{2}by)-F_{X_{\mu,b}}(\mu-\sqrt{2}by)\right]\\
&=\inf_{\mu,b}\left[1-\frac{1}{2}e^{-\frac{(\mu+\sqrt{2}by)-\mu}{b}}-
\frac{1}{2}e^{\frac{(\mu-\sqrt{2}by)-\mu}{b}}\right]\\
&=1-e^{-\sqrt{2}y}.
\end{align*}
\indent
(iii) Since $b>0$ and $y>0$, we have
\begin{align*}
T(y)&=\inf_{\mu,b}P\left(|X_{\mu,b}-\mu|\geq \sqrt{2}by\right)\\
&=\inf_{\mu,b}\left[1-F_{X_{\mu,b}}(\mu+\sqrt{2}by)+F_{X_{\mu,b}}(\mu-\sqrt{2}by)\right]\\
&=\inf_{\mu,b}\left[1-\left(1-\frac{1}{2}e^{-\frac{(\mu+\sqrt{2}by)-\mu}{b}}\right)+
\frac{1}{2}e^{\frac{(\mu-\sqrt{2}by)-\mu}{b}}\right]\\
&=e^{-\sqrt{2}y}.
\end{align*}
The proof is complete.\hfill\fbox

%
%

\section{Proof of Theorem \ref{main-thm2}}\label{S3}\setcounter{equation}{0}

Define the hypergeometric function $F(a, b; c; z)$ by (cf. \cite[P. 45]{Rai60})
\[
F(a, b;c;z)
:= \sum_{j=0}^{ + \infty}{ \frac{ (a)_j (b)_j}{(c)_j} \cdot \frac{ z^j }{ j !}}, |z| < 1,
\]
where $(\alpha)_j := \alpha ( \alpha + 1) \cdots (  \alpha + j - 1 ) $ for $ j \ge 1 $, and $ ( \alpha )_0 := 1 $ for $  \alpha \ne 0$. Then the  distribution function $F_v$ of $ X_{v} $ can be expressed by
\begin{align}\label{3.1}
F_v(x)=\frac{1}{2}+x\Gamma\left(\frac{v+1}{2}\right)\frac{F \left(\frac{1}{2}, \frac{v+1}{2};\frac{3}{2};-\frac{x^2}{v} \right)}{\sqrt{v\pi}\Gamma({\frac{v}{2}})},\ x\in \mathbb{R}.
\end{align}
\indent By the fact that $E(X_{v})=0$ and (\ref{3.1}), we have
$$
P(X_v\leq yE(X_v))=P(X_v\leq 0)=F_v(0)=\frac{1}{2}, \forall y>0.
$$
It follows that for any $y>0, C(y)=\frac{1}{2}$, and thus  (i) holds. In the following, we focus on the proofs of (ii) and (iii).

By the symmetry of the density function and (\ref{3.1}),  we have
\begin{align*}
T(y)
&=P \left({|X_v|}\leq y\sqrt{\frac{v}{v-2}} \right)\\
&=2F{_v} \left(y\sqrt{\frac{v}{v-2}} \right)-1\\
&=2y\sqrt{\frac{v}{v-2}}\frac{\Gamma(\frac{v+1}{2})}{\sqrt{v\pi}\Gamma(\frac{v}{2}) }{F \left(\frac{1}{2}, \frac{v+1}{2};\frac{3}{2};-\frac{y^2}{v-2} \right)}
\\
&
=:J{_v}, \\
H(y)
&= 1 - P \left({|X_v|}\leq y\sqrt{\frac{v}{v-2}} \right)=1 - J{_v}.
\end{align*}

Consider the weak monotonicity of $J_v$:
\begin{align*}{\quad\quad\quad\quad}
\frac{J{_{v+2}}}{J{_v}}<1
&\Leftrightarrow
\frac
{2y\sqrt{\frac{v+2}{v}}\frac{\Gamma(\frac{v+3}{2})}{\sqrt{(v+2)\pi}\Gamma(\frac{v+2}{2}) }{F \left (\frac{1}{2}, \frac{v+3}{2};\frac{3}{2};-\frac{y^2}{v} \right )}}
{2y\sqrt{\frac{v}{v-2}}\frac{\Gamma(\frac{v+1}{2})}{\sqrt{v\pi}\Gamma(\frac{v}{2}) }{F \left(\frac{1}{2}, \frac{v+1}{2};\frac{3}{2};-\frac{y^2}{v-2} \right)}}<1 \\
&\Leftrightarrow
\frac{(v+1){(v-2)}^{\frac{1}{2}}}{{v}^{\frac{3}{2}}} \frac{F \left (\frac{1}{2}, \frac{v+3}{2};\frac{3}{2};-\frac{y^2}{v} \right )}{F \left (\frac{1}{2}, \frac{v+1}{2};\frac{3}{2};-\frac{y^2}{v-2} \right )}<1
\\
&\Leftrightarrow
F \left (\frac{1}{2}, \frac{v+3}{2};\frac{3}{2};-\frac{y^2}{v} \right )
<\frac{{v}^{\frac{3}{2}}}{(v+1){(v-2)}^{\frac{1}{2}}}F \left (\frac{1}{2}, \frac{v+1}{2};\frac{3}{2};-\frac{y^2}{v-2} \right ).
\end{align*}
By the contiguous function relation of the hypergeometric function (cf. \cite[P. 71, 21(13)]{Rai60}) and the fact that $ F(a, b;a;z) = {{(1-z)}^{-b}}$ (cf. \cite[P. 47, Line -4]{Rai60}), we have
\begin{align*}
{\quad\quad\quad\quad}
\frac{v+1}{2}F \left( \frac{1}{2}, \frac{v+3}{2};\frac{3}{2};-\frac{y^2}{v} \right)
&=\frac{v}{2}F \left( \frac{1}{2}, \frac{v+1}{2};\frac{3}{2};-\frac{y^2}{v} \right) +\frac{1}{2}F \left( \frac{1}{2}, \frac{v+1}{2};\frac{1}{2};-\frac{y^2}{v} \right) \\
&=\frac{v}{2}F \left( \frac{1}{2}, \frac{v+1}{2};\frac{3}{2};-\frac{y^2}{v} \right) +\frac{1}{2}{ \left( {\frac{v}{v+y^2} } \right)^{\frac{v+1}{2}}},
\end{align*}
which implies that
\begin{align*}
F \left (\frac{1}{2}, \frac{v+3}{2};\frac{3}{2};-\frac{y^2}{v} \right )
&=\frac{v}{v+1}F \left (\frac{1}{2}, \frac{v+1}{2};\frac{3}{2};-\frac{y^2}{v} \right )+\frac{1}{v+1}{ \left( {\frac{v}{v+y^2} } \right)^{\frac{v+1}{2}}}.
\end{align*}
Then we have
\begin{align}
&\frac{J{_{v+2}}}{J{_v}}<1\nonumber\\
&\Leftrightarrow
\frac{v}{v+1}F \left( \frac{1}{2}, \frac{v+1}{2};\frac{3}{2};-\frac{y^2}{v} \right) +\frac{1}{v+1}{ \left( {\frac{v}{v+y^2}} \right) }^{\frac{v+1}{2}}
 <\frac{{v}^{\frac{3}{2}}}{(v+1){(v-2)}^{\frac{1}{2}}}F \left( \frac{1}{2}, \frac{v+1}{2};\frac{3}{2};-\frac{y^2}{v-2} \right) \nonumber\\
&\Leftrightarrow
F \left( \frac{1}{2}, \frac{v+1}{2};\frac{3}{2};-\frac{y^2}{v} \right) +\frac{1}{v}{ \left( {\frac{v}{v+y^2}} \right)^{\frac{v+1}{2}}}
 <{ \left( \frac{v}{v-2} \right) }^{\frac{1}{2}}F \left( \frac{1}{2}, \frac{v+1}{2};\frac{3}{2};-\frac{y^2}{v-2} \right) \nonumber\\
&\Leftrightarrow
\int_{0}^{y}{{f}{_v{(x)}}} \, dx+
\frac{y\Gamma(\frac{v+1}{2})}{\sqrt{v\pi }\Gamma(\frac{v}{2})}  \cdot \frac{1}{v} \cdot { \left( {\frac{v}{v+y^2}} \right)^{\frac{v+1}{2}}}
<\int_{0}^{y \sqrt{\frac{v}{v-2}}} {{f}{_v{(x)}}} \, dx\nonumber\\
&\Leftrightarrow
\frac{y\Gamma(\frac{v+1}{2})}{\sqrt{v\pi} \Gamma(\frac{v}{2})}  \cdot\frac{1}{v} \cdot { \left( {\frac{v}{v+y^2}} \right)^{\frac{v+1}{2}}}
<\int_{y}^{y \sqrt{\frac{v}{v-2}}}{{f}{_v{(x)}}} \, dx\nonumber\\
&\Leftrightarrow
\frac{y}{v} { \left( {\frac{v}{v+y^2}} \right)^{\frac{v+1}{2}}}
<\int_{y}^{y \sqrt{\frac{v}{v-2}}}{ \left( \frac{v}{v+x^2} \right)^{\frac{v+1}{2}}} \, dx \nonumber\\
&\Leftrightarrow
1<\frac{v}{y}\int_{y}^{y \sqrt{\frac{v}{v-2}}}{ \left( \frac{v+y^2}{v+x^2} \right)^{\frac{v+1}{2}}}dx\nonumber\\
&\quad\quad=v \int_{1}^{ \sqrt{\frac{v}{v-2}}}{ \left( \frac{v+y^2}{v+y^2x^2} \right)^{\frac{v+1}{2}}} \, dx=:G(v, y).\label{3.2}
\end{align}
Similarly, we have
\[
G(v, y)<1
\Leftrightarrow
\frac{J{_{v+2}}}{J{_v}}>1.
\]
\indent In the following, we  consider any real number $v \ge 3$. For convenience, for $v\geq 3, y>0$ and $x\in [1,\sqrt{\frac{v}{v-2}}]$, define
\begin{align}\label{3.3}
u(v):=\sqrt{\frac{v}{v-2}}, \quad L(v):=u(v)-1, \quad a:=y ^2,
\end{align}
and
\begin{align}\label{3.4}
f(v, x):= \left( \frac{v+a}{v+ax^2} \right)^{\frac{v+1}{2}}.
\end{align}
Then $u(v)>1$ and $L(v)>0$.

By the properties of $t$ distribution, we have (cf. \cite{SHS24b})
$$
\lim_{v\to+\infty}F_v(x)=\Phi(x),
$$
where $\Phi(x) $ is the distribution function of the standard normal distribution.


\subsection{The case $ y \in (0, 1] $}

For $ y \in (0, 1] $ and $ x \ge 1$, we have
$$
\frac{v+y^2}{v+y^2x^2}
= \frac{\frac{v}{y^2}+1}{\frac{v}{y^2}+x^2}
= 1 - \frac{x^2-1}{\frac{v}{y^2}+x^2}
\ge 1 - \frac{x^2-1}{v+x^2}
= \frac{v+1}{v+x^2}.
$$
By the proof of \cite[Theorem 1.3 (vii)]{SHS24b},  we know that
$$
v \int_{1}^{ \sqrt{\frac{v}{v-2}}}{ \left( \frac{v+1}{v+x^2} \right)^{\frac{v+1}{2}}}dx >1,\ \forall v \ge 3.
$$
Hence, for all $ y \in (0, 1]$,
$$
G(v,y) > 1,\ \forall v \ge 3,
$$
which implies that for all $v \ge 3, \frac{J_{v+2}}{J_v} <1$. It follows that
\begin{align*}
&\begin{aligned}
T(y)
&=\inf_{v \geq 3}{P \left({|X-E(X)|}\leq y\sqrt{Var(X)} \right)}
\\
&= \lim_{v \to +\infty}\left(2F{_v} \left( y\sqrt{\frac{v}{v-2}} \right)-1\right)
= 2\Phi(y)-1,
\end{aligned}
\\
&\begin{aligned}
H(y)
&=\inf_{v \geq 3}{P \left({|X-E(X)|} \ge  y\sqrt{Var(X)} \right)}
\\
&= \min_{v=3,4}{ \left\{ 2 - 2F{_v} \left( y\sqrt{\frac{v}{v-2}} \right) \right\} }.
\end{aligned}
\end{align*}
The proof in this case is complete.\hfill\fbox

\subsection{The case $ y \in (1, \sqrt{3}) \cup (\sqrt{3},+\infty) $}

 At first, we give four lemmas, whose proofs will be postponed to next section. Recall that $L(v), f(v,x)$ and $G(v,y)$ are defined by (\ref{3.3}), (\ref{3.4}) and (\ref{3.2}), respectively.

\begin{lemma}\label{lemma1}
Let \(C_0 = 3 \), \(V_0 = 100\). For \( v \ge V_0 \), we have
\[ L(v) = \frac{1}{v} + \frac{3}{2v^2} + R_0(v) , |R_0(v)| < \frac{C_0}{v^3},
\]
where $ R_0(v) := \frac{\phi_1^{(3)}(\xi_1)}{6v^3}, \phi_1(x) := (1 - 2x)^{-\frac{1}{2}}, \xi_1 \in (0,\frac{1}{v}). $
\end{lemma}


Recall that  $C_1,C_2,C_3, V_1, V_2$ and $V_3$ are defined by (\ref{1.1}).

\begin{lemma}\label{lemma2}
For \( v \ge V_1\), we have
\[
\ln f(v, x)= E_0(x) + \frac{1}{v} E_1(x) + \frac{1}{v^2} R_1(v, x),\  |R_1(v, x)| < C_1 ,
\]
where
\begin{align}
E_0(x)&:=\frac{a}{2} (1-x^2),\ E_1(x):=\frac{a^2}{4} (x^4-1)-\frac{a}{2} (x^2-1),\label{3.5}\\
R_1(v, x)&:=\frac{a^2}{4} (x^4-1) + \frac{v+1}{2} \left( \frac{\phi_2 ^{(3)}(\xi_2)}{6} \cdot \frac{a^3}{v}-\frac{\phi_2 ^{(3)}(\xi_3)}{6} \cdot \frac{a^3 x^6}{v} \right),\label{3.6}\\
\phi_2(x)&:= \ln(1+x),\ \xi_2 \in (0, \frac{a}{v}),\ \xi_3 \in (0, \frac{ax^2}{v}).\nonumber
\end{align}
\end{lemma}

\begin{lemma}\label{lemma3}
 For \( v\ge V_2 \), we have
\[
f(v, x) = e^{E_0(x)} \left( 1+\frac{E_1(x)}{v}  + \frac{R_2(v, x)}{v^2} \right),\ |R_2(v, x)| < C_2,
\]
where
\begin{align}\label{lem-3.3-a}
R_2(v, x)&:= \frac{2R_1(v, x) + E_1^2(x)}{2} + \frac{R_1(v, x) \cdot E_1(x) }{v} + \frac{R_1^2(v, x)}{2v^2} + \frac{e^{\xi_4}}{6} v^2 z^3,
\end{align}
and $E_1(x)$ is defined in (\ref{3.5}),  $R_1(v,x)$ is defined by (\ref{3.6}), \(z := \frac{1}{v} E_1(x) + \frac{1}{v^2} R_1(v, x)\), \(\xi_4\) lies between 0 and \(z\).

\end{lemma}

\begin{lemma}\label{lemma4}
For  \( v\ge V_3 \), we have
\[
G(v, y) = 1 + \frac{3-a}{2v} + R_G(v),  |R_G(v)| <  \frac{C_{3}}{v^2},
\]
where
\begin{align*}
R_G(v) &:=A_1 + A_2+ R_3(v) + R_4(v),\\
A_1 &:= v \int_0^{L(v)} e^{E_0(1+m)} \frac{E_1(1+m)}{v}\, dm,\\
 A_2 &:= v \int_0^{L(v)} e^{E_0(1+m)}\frac{R_2(v, 1+m)}{v^2}dm,\\
R_3(v) &:= v R_0(v)-\left( \frac{9 a}{8 v^3} + \frac{av R_0^2(v)}{2} + \frac{3 a}{2 v^2} + a R_0(v) + \frac{3 a R_0(v)}{2 v} \right),\\
R_4(v) &:= \int_0^{L(v)}\left[\frac{e^{\xi_5}}{2} a^2m^2 + m^2 \left( 1 - am + \frac{e^{\xi_5}}{2} a^2m^2 \right) \left( - \frac{a}{2} + \frac{e^{\xi_6}}{8} {a^2 m^2} \right)\right]dm,
\end{align*}
$E_0(\cdot), E_1(\cdot)$ are defined by (\ref{3.5}), $R_2(\cdot,\cdot)$ is defined by (\ref{lem-3.3-a}), $R_0(v)$ is defined in Lemma \ref{lemma1},  $\xi_5 \in (-am, 0)$ and $\xi_6 \in (-\frac{am^2}{2},0)$.

\end{lemma}

\noindent {\bf Proof of Theorem \ref{main-thm2} (ii)(iii) for $ y \in (1, \sqrt{3}) \cup (\sqrt{3},+\infty)$.} We decompose the proof into two cases: $1<y<\sqrt{3}$ and $y>\sqrt{3}$. By (\ref{1.1}), we have
$v_0(y)=\max\{V_3, \frac{2C_3}{|a-3|}+1\}$.\\
\indent (1) If $1<y<\sqrt{3}$, then $1<a=y^2<3$.  By Lemma \ref{lemma4}, we get that for all $v\ge v_0(y)$,
\begin{align*}
G(v, y)
&= 1 + \frac{3-a}{2v} + R_G(v)
\\
&> 1 + \frac{3-a}{2v} - \frac{C_{3}}{v^2} = 1 + \frac{(3-a)v-2C_{3}}{2v^2} > 1 ,
\end{align*}
which implies that for all $v\ge v_0(y), \frac{J_{v+2}}{J_v} <1$. It follows that
\begin{align*}
&\begin{aligned}
T(y)
&=\inf_{v \geq 3}{P \left({|X-E(X)|}\leq y\sqrt{Var(X)} \right)}
\\
&= \min{ \left\{ \min_{3 \le v \le [v_0(y)]+3} \left\{ 2F{_v} \left( y\sqrt{\frac{v}{v-2}} \right)-1 \right\}, \lim_{v \to +\infty}\left(2F{_v} \left( y\sqrt{\frac{v}{v-2}} \right)-1\right) \right\}  }
\\
&= \min{ \left\{ {\min_{3 \le v \le [v_0(y)]+3}  \left\{ 2F{_v} \left( y\sqrt{\frac{v}{v-2}} \right) -1 \right\} }, 2\Phi(y)-1 \right\} },
\end{aligned}
\\
&\begin{aligned}
H(y)
&=\inf_{v \geq 3}{P \left({|X-E(X)|} \ge  y\sqrt{Var(X)} \right)}
\\
&= \min_{3 \le v \le [v_0(y)]+3}{ \left\{ 2 - 2F{_v} \left( y\sqrt{\frac{v}{v-2}} \right) \right\} }.
\end{aligned}
\end{align*}

(2) If $y>\sqrt{3}$, then $a=y^2>3$. By Lemma \ref{lemma4}, we get that for all $v\ge v_0(y)$,
\begin{align*}
G(v, y)
&= 1 - \frac{a-3}{2v} + R_G(v)
\\
&< 1 - \frac{a-3}{2v} + \frac{C_{3}}{v^2} = 1 - \frac{(a-3)v-2C_{3}}{2v^2} < 1,
\end{align*}
which implies that for all $v\ge v_0(y), \frac{J_{v+2}}{J_v}>1$. It follows that
\begin{align*}
&\begin{aligned}
T(y)
&=\inf_{v \geq 3}{P \left({|X-E(X)|}\leq y\sqrt{Var(X)} \right)}
\\
&= \min_{3 \le v \le [v_0(y)]+3}{ \left\{ 2F{_v} \left( y\sqrt{\frac{v}{v-2}} \right) -1 \right\}},
\end{aligned}
\\
&\begin{aligned}
H(y)
&=\inf_{v \geq 3}{P \left({|X-E(X)|} \ge  y\sqrt{Var(X)} \right)}
\\
&= \min{ \left\{ {\min_{3 \le v \le [v_0(y)]+3} { \left\{ 2 - 2F{_v} \left( y\sqrt{\frac{v}{v-2}} \right)  \right\} } }, \lim_{ v \to +\infty }\left(2 - 2F{_v} \left( y\sqrt{\frac{v}{v-2}} \right)\right) \right\}  }
\\
&= \min{ \left\{ {\min_{3 \le v \le [v_0(y)]+3} { \left\{ 2 - 2F{_v} \left( y\sqrt{\frac{v}{v-2}} \right)  \right\} } }, 2 - 2\Phi(y) \right\}}.
\end{aligned}
\end{align*}
The proof in this case is complete.\hfill\fbox

\subsection{The case $ y = \sqrt{3} $}

Now $a=y^2=3$. Thus
\[
G(v, y) = 1 + \frac{3-a}{2v} + R_G(v)=1+R_G(v).
\]
Hence we need retain higher order item to determine whether $G(v,y)>1$ or $G(v,y)<1$.  Recall that $ \bar{v}_0(\sqrt{3}) = 1318.4$.

\begin{lemma}
For  \( v\ge \bar{v}_0 (\sqrt{3})\), we have $ G \left( v, \sqrt{3} \right)>1$.
\label{lemma5}
\end{lemma}

The proof of the above lemma will be  given in next section.

\medskip

\noindent {\bf Proof of Theorem 1.2(ii)(iii) for $ y = \sqrt{3}$.} By Lemma \ref{lemma5}, we know that for all $v\geq \bar{v}_0(\sqrt{3})$, $G \left( v, \sqrt{3} \right)>1$, which implies that for all $v\geq \bar{v}_0(\sqrt{3}), \frac{J_{v+2}}{J_v}<1$. It follows that
\begin{align*}
&\begin{aligned}
T(\sqrt{3})
&=\inf_{v \geq 3}{P \left({|X-E(X)|}\leq \sqrt{3Var(X)} \right)}
\\
&= \min{ \left\{ {\min_{3 \le v \le [\bar{v}_0(\sqrt{3})]+3}  \left\{ 2F{_v} \left(\sqrt{\frac{3v}{v-2}} \right)-1 \right\} } , {\lim_{ v \to +\infty}  \left(2F{_v} \left(\sqrt{\frac{3v}{v-2}} \right)-1\right) } \right\} }
\\
&= \min{ \left\{ {\min_{3 \le v \le [\bar{v}_0(\sqrt{3})]+3}  \left\{ 2F{_v} \left(\sqrt{\frac{3v}{v-2}} \right)-1 \right\} } , 2\Phi(\sqrt{3})-1 \right\} },
\end{aligned}
\\
&\begin{aligned}
H(\sqrt{3})
&=\inf_{v \geq 3}{P \left({|X-E(X)|} \ge  \sqrt{3Var(X)} \right)}
\\
&= \min_{3 \le v \le [\bar{v}_0(\sqrt{3})]+3}{ \left\{ 2 - 2F{_v} \left( \sqrt{\frac{3v}{v-2}} \right) \right\} } .
\end{aligned}
\end{align*}
The proof in this case is complete.\hfill\fbox



\section{Proofs of Lemmas 3.1-3.5}

Notice that  $a=y^2$. Then as to the proofs of Lemmas 3.1-3.4, we have that  $a\geq 1$ and $a\neq 3$.

\subsection{Proof of Lemma \ref{lemma1}}

 Recall that \(\phi_1(t) = (1 - 2t)^{-\frac{1}{2}}\). Performing a Taylor's expansion at \(t = 0\) with the Lagrange's remainder term, we get
\[
\phi_1(t) = 1 + t + \frac{3}{2} t^2 + \frac{\phi_1^{(3)}(\xi_1)}{6} t^3,
\]
where  $\xi_1 \in (0, t),  \phi_1^{(3)}(t) = 15(1 - 2t)^{-\frac{7}{2}}.$

When $0 \le t \le \frac{1}{100}$, we have $|\phi_1^{(3)}(\xi_1)| \le 15 \cdot {\frac{50}{49}}^{\frac{7}{2}}$ and $v:=\frac{1}{t}\ge 100$. Thus for $v\geq 100$, we have
\begin{align*}
L(v)&=\phi_1(1/v)-1=\frac{1}{v}+\frac{3}{2v^2}+R_0(v),\\
|R_0(v)|
&= \frac{|\phi_1^{(3)}(\xi_1)|}{6} \cdot \frac{1}{v^3} \le \frac{15 \cdot {\frac{50}{49}}^{\frac{7}{2}}}{6} \cdot \frac{1}{v^3}<\frac{3}{v^2}=\frac{C_0}{v^3}.
\end{align*}
The proof is complete.\hfill\fbox

\subsection{Proof of Lemma \ref{lemma2}}

By the definition  of $f(v, x)$ (cf. (\ref{3.4})), we have
\begin{align}\label{Proof-Lemma-3.2-a}
\ln f(v, x) = \frac{v+1}{2} \left[ \ln\left(1 + \frac{a}{v}\right) - \ln\left(1 + \frac{a x^2}{v}\right)\right].
\end{align}
\indent Performing a Taylor's expansion of \(\phi_2(t):=\ln(1+t)\) at \(t = 0\) with the Lagrange's remainder term, we get that for $t>0$,
\begin{align}\label{Proof-Lemma-3.2-b}
\ln(1+t) = t - \frac{t^2}{2} + \frac{\phi_2^{(3)}(\xi)}{6} t^3,
\end{align}
where  $ \xi \in (0, t), \phi_2^{(3)}(t) = \frac{2}{(1+t)^3}$.
Let \(t_1 := \frac{a}{v},  t_2 := \frac{a x^2}{v}\). Then $t_1>0, t_2>0$.

By (\ref{Proof-Lemma-3.2-a}) and (\ref{Proof-Lemma-3.2-b}), we get
\begin{align}\label{Proof-Lemma-3.2-c}
\ln f(v, x)
&= \frac{v+1}{2} \left[\left( \frac{a}{v}-\frac{a^2}{2v^2}+\frac{\phi_2 ^{(3)}(\xi_2)}{6} \cdot \frac{a^3}{v^3} \right) - \left( \frac{ax^2}{v}-\frac{a^2 x^4}{2v^2}+\frac{\phi_2 ^{(3)}(\xi_3)}{6} \cdot \frac{a^3 x^6}{v^3} \right) \right]\nonumber\\
&=\frac{v+1}{2} \left [ \frac{a}{v}(1-x^2) + \frac{a^2}{2v^2}(x^4-1) + \left( \frac{\phi_2 ^{(3)}(\xi_2)}{6} \cdot \frac{a^3}{v^3}-\frac{\phi_2 ^{(3)}(\xi_3)}{6} \cdot \frac{a^3 x^6}{v^3} \right) \right ]\nonumber\\
&=\frac{a}{2} (1-x^2) +\frac{1}{v} \left [ \frac{a^2}{4} (x^4-1)-\frac{a}{2} (x^2-1) \right ]\nonumber\\
&  \quad \quad + \frac{1}{v^2} \left [ \frac{a^2}{4} (x^4-1) + \frac{v+1}{2} \left( \frac{\phi_2 ^{(3)}(\xi_2)}{6} \cdot \frac{a^3}{v}-\frac{\phi_2 ^{(3)}(\xi_3)}{6} \cdot \frac{a^3 x^6}{v}  \right) \right ]\nonumber\\
&= E_0(x) + \frac{1}{v} E_1(x) + \frac{1}{v^2} R_1(v, x),
\end{align}
where  $\xi_2 \in (0, t_1), \xi_3 \in (0, t_2).$

For  $v \ge V_1=\max\{100,8a\}$, and  $1 \le x\leq \sqrt{\frac{v}{v-2}} \le \sqrt{2}$, we have that $0<t_1,t_2<1$ and thus $\frac{1}{4}<\phi_2^{(3)}(\xi_2),\phi_2^{(3)}(\xi_3)<2$. Hence
\begin{align}\label{Proof-Lemma-3.2-d}
|R_1(v, x)|
&\le \frac{a^2}{4} \left [(\sqrt{2})^4-1 \right] + \frac{v+1}{2} \left [ \frac{2}{6} \cdot \frac{a^3}{v}+\frac{2}{6} \cdot \frac{a^3 (\sqrt{2})^6}{v} \right]\nonumber\\
&=\frac{3}{4}a^2 + \frac{3}{2}a^3 \cdot \frac{v+1}{v}\nonumber\\
& < 2a^3 + \frac{3}{4}a^2 =C_1.
\end{align}
The proof is complete. \hfill\fbox

\subsection{Proof of Lemma \ref{lemma3}}

By Lemma \ref{lemma2}, we have
\begin{align}\label{proof-lemma-3.3-a}
f(v, x)=e^{E_0(x)} e^{\frac{1}{v} E_1(x) + \frac{1}{v^2} R_1(v, x)},
\end{align}
where $E_0(x), E_1(x)$ and $R_1(v,x)$ are defined in Lemma \ref{lemma2}.

Recall that
\begin{align}\label{proof-lemma-3.3-b}
z=\frac{1}{v} E_1(x) + \frac{1}{v^2} R_1(v, x).
\end{align}
By the Taylor's expansion, we have
\begin{align}\label{proof-lemma-3.3-bb}
e^z = 1 + z + \frac{z^2}{2} + \frac{e^{\xi_4}}{6} z^3,
\end{align}
 where \(\xi_4 \) lies between 0 and $z$. Then by  (\ref{proof-lemma-3.3-b}) and (\ref{lem-3.3-a}),  we get
\begin{align}\label{proof-lemma-3.3-bb-1}
e^z &= 1 + \left( \frac{1}{v} E_1(x) + \frac{1}{v^2} R_1(v, x) \right) + \frac{1}{2}{ \left( \frac{1}{v} E_1(x) + \frac{1}{v^2} R_1(v, x) \right)^2} + \frac{e^{\xi_4}}{6} z^3\nonumber\\
&= 1 + \frac{E_1(x)}{v}  + \frac{2R_1(v, x) + E_1^2(x)}{2v^2} + \frac{R_1(v, x) \cdot E_1(x) }{v^3}
+ \frac{R_1^2(v, x)}{2v^4} + \frac{e^{\xi_4}}{6} z^3\nonumber\\
&=1 + \frac{E_1(x)}{v}+\frac{1}{v^2}\left(\frac{2R_1(v, x) + E_1^2(x)}{2} + \frac{R_1(v, x) \cdot E_1(x) }{v} + \frac{R_1^2(v, x)}{2v^2} + \frac{e^{\xi_4}}{6} v^2 z^3\right)\nonumber\\
&=1 + \frac{E_1(x)}{v}+\frac{R_2(v,x)}{v^2}.
\end{align}

For  $v \ge V_1=\max\{100,8a\}$, we have  $1 \le x\leq \sqrt{\frac{v}{v-2}} \le \sqrt{2}$. Since $E_1(x) =\frac{a^2}{4} (x^4-1)-\frac{a}{2} (x^2-1)$, we have
\[
E_1'(x) = a^2x^3-ax = a^2x(x^2-\frac{1}{a}) > 0,
\]
which implies that $E_1(x)$ is strictly increasing on $ [1,\sqrt{2}] $.  Thus
\begin{align}\label{proof-lemma-3.3-c}
0 \le E_1(x)\le \frac{3}{4}a^2-\frac{a}{2} < \frac{3}{4}a^2.
\end{align}

We have
\begin{align}\label{proof-lemma-3.3-d}
v^2 z^3
=v^2 \left( \frac{1}{v} E_1(x) + \frac{1}{v^2} R_1(v, x) \right)^3
= \left( \frac{1}{v^{\frac{1}{3}}} E_1(x) + \frac{1}{v^{\frac{4}{3}}} R_1(v, x) \right)^3.
\end{align}

Recall that $V_2=\max \left\{ V_1, \frac{3}{2}a^2, \sqrt{2C_1} \right\}$. For $v\geq V_2$, by (\ref{proof-lemma-3.3-c}) and the fact that $|R_1(v,x)|\leq C_1$, we get that
$$
\frac{E_1(x)}{v}\leq \frac{1}{2},\  \frac{|R_1(v,x)|}{v^2}\leq \frac{1}{2},
$$
which together with (\ref{proof-lemma-3.3-b}) implies that for $v\geq V_2$, we have $|z|\leq 1$, and thus for the $\xi_4$ in (\ref{proof-lemma-3.3-bb}), we have
\begin{align}\label{proof-lemma-3.3-e}
|\xi_4|< 1.
\end{align}

For $v\geq V_2$, by the fact that $|R_1(v,x)|\leq C_1$, (\ref{proof-lemma-3.3-c}),  (\ref{proof-lemma-3.3-d}), and (\ref{proof-lemma-3.3-e})),  we get that
\begin{align*}
|R_2(v, x)|&=\left|\frac{2R_1(v, x) + E_1^2(x)}{2} + \frac{R_1(v, x) \cdot E_1(x) }{v} + \frac{R_1^2(v, x)}{2v^2} + \frac{e^{\xi_4}}{6} v^2 z^3\right|\\
&\leq |R_1(v, x)| + \frac{E_1^2(x)}{2} + \frac{|R_1(v, x)| \cdot E_1(x) }{v} + \frac{R_1^2(v, x)}{2v^2} + \left| \frac{e^{\xi_4}}{6} v^2 z^3 \right| \\
&< C_1 + \frac{1}{2} \left( \frac{3}{4}a^2 \right)^2 + \frac{C_1 \cdot \frac{3}{4}a^2}{V_2} + \frac{C_1^2}{2V_2^2}+ \frac{e}{6}  \left( \frac{\frac{3}{4}a^2}{V_2^{\frac{1}{3}}}  + \frac{C_1}{V_2^{\frac{4}{3}}} \right)^3\\
&\le C_1 + \frac{9}{32}a^4 + \frac{C_1 \cdot \frac{3}{4}a^2}{\frac{3}{2}a^2} + \frac{C_1^2}{2 \left( \sqrt{2C_1} \right)^2}+ \frac{e}{6}  \left[\frac{\frac{3}{4}a^2}{ \left( \frac{3}{2}a^2 \right)^{\frac{1}{3}}}  + \frac{C_1}{ \left( \sqrt{2C_1} \right)^{\frac{4}{3}}} \right] ^3
\\
&\le C_1 + \frac{9}{32}a^4 + \frac{1}{2}C_1 + \frac{1}{4}C_1 + \frac{1}{2} \left( a^{\frac{4}{3}} + C_1^{\frac{1}{3}} \right)^3
\\
&= \frac{9}{4}C_1 + \frac{25}{32}a^4 + \frac{3}{2} \left( C_1^{\frac{2}{3}} \cdot a^{\frac{4}{3}} + C_1^{\frac{1}{3}} \cdot a^{\frac{8}{3}} \right)= C_2.
\end{align*}
Hence  for all  \( v\ge V_2 \) ,
\[ f(v, x) = e^{E_0(x)} \left( 1+\frac{E_1(x)}{v}  + \frac{R_2(v, x)}{v^2} \right),\ |R_2(v, x)| < C_2.
\]
The proof is complete.\hfill\fbox

\subsection{Proof of Lemma \ref{lemma4}}

We decompose the proof into four steps. Recall that $1\leq x\leq \sqrt{\frac{v}{v-2}}=u(v)$ and $L(v)=u(v)-1$.

\textbf{Step 1.} Let $m:=x-1$. Then \( 0 \le m \le L(v)\). By Lemma \ref{lemma3} and the definitions of $A_1,A_2$ in Lemma \ref{lemma4}, we get that for all $v\geq V_2$,
\begin{align}\label{proof-lemma3.4-a}
G(v, y) &=v \int_1^{u(v)} f(v, x) \, dx = v \int_0^{L(v)}f(v, 1+m)dm\nonumber\\
&= v \int_0^{L(v)} e^{E_0(1+m)} \left( 1 + \frac{E_1(1+m)}{v} + \frac{R_2(v, 1+m)}{v^2} \right)dm\nonumber\\
&= v \int_0^{L(v)} e^{E_0(1+m)}dm + A_1 + A_2.
\end{align}
In the following, we consider the first term $v \int_0^{L(v)} e^{E_0(1+m)}dm$.

By  (\ref{3.5}), we have
\begin{align}
E_0(1+m)
&= \frac{a}{2} [1-(1+m)^2] = -am-\frac{am^2}{2},\label{proof-lemma3.4-b}\\
E_1(1+m)&= \frac{a^2}{4} \left[ (1+m)^4-1 \right]-\frac{a}{2} \left[ (1+m)^2-1) \right]\nonumber\\
&= (a^2-a)m+ \left( \frac{3}{2}a^2-\frac{a}{2} \right) m^2+a^2 m^3+\frac{a^2 m^4}{4},\label{proof-lemma3.4-c}
\end{align}
and thus
\begin{align*}
v \int_0^{L(v)} e^{E_0(1+m)} \, dm
&= v \int_0^{L(v)} e^{-am} \cdot e^{-\frac{am^2}{2}} \, dm .
\end{align*}

By (\ref{1.1}) and $a=y^2$, we know that $V_3=\max\left\{V_2, 2+\frac{2a}{1+2\sqrt{a}}, 2+\frac{2a^2}{1+2a}\right\}$. For $v\geq V_3$, one can check that $0\leq L(v)<1$ and $0\leq aL(v), aL^2(v)\leq 1$. For $v\geq V_3$ and $0\leq m\leq L(v)$, we have
\begin{align*}
e^{-am} &= 1 - am + \frac{e^{\xi_5}}{2} (-am)^2 =  1 - am + \frac{e^{\xi_5}}{2} a^2m^2, \quad \xi_5 \in (-am, 0)\subset (-1,0)\\
e^{-\frac{am^2}{2}} &= 1 - \frac{am^2}{2} + \frac{e^{\xi_6}}{2} \left( -\frac{am^2}{2} \right)^2
=1 - \frac{am^2}{2} + \frac{e^{\xi_6}}{8} {a^2 m^4},  \quad \xi_6 \in (-\frac{am^2}{2}, 0)\subset (-1,0),
\\
v \int_0^{L(v)} e^{E_0(1+m)} \, dm &= v \int_0^{L(v)} \left( 1 - am + \frac{e^{\xi_5}}{2} a^2m^2 \right) \cdot \left( 1 - \frac{am^2}{2} + \frac{e^{\xi_6}}{8} {a^2 m^4} \right) \, dm
\\
&= v \int_0^{L(v)}\left[1 - am + \frac{e^{\xi_5}}{2} a^2m^2 + \left( 1 - am + \frac{e^{\xi_5}}{2} a^2m^2 \right) \left( - \frac{am^2}{2} + \frac{e^{\xi_6}}{8} {a^2 m^4} \right)\right]dm\\
&= v \int_0^{L(v)} (1 - am) dm\\
& \quad + v \int_0^{L(v)} \left[\frac{e^{\xi_5}}{2} a^2m^2 + m^2 \left( 1 - am + \frac{e^{\xi_5}}{2} a^2m^2 \right)\left( - \frac{a}{2} + \frac{e^{\xi_6}}{8} {a^2 m^2} \right)\right]dm
\\
&= v \left( L(v) - \frac{a}{2} L^2(v) \right)
\\
& \quad + v \int_0^{L(v)}\left[\frac{e^{\xi_5}}{2} a^2m^2 + m^2 \left( 1 - am + \frac{e^{\xi_5}}{2} a^2m^2 \right)\left( - \frac{a}{2} + \frac{e^{\xi_6}}{8} {a^2 m^2} \right)\right]dm.
\end{align*}

By Lemma \ref{lemma1} and the definition of $R_3(v)$ in Lemma \ref{lemma4}, we get
\begin{align*}
v \left( L(v) - \frac{a}{2} L^2(v) \right)
&= v \left[ \left (\frac{1}{v} + \frac{3}{2v^2} + R_0(v) \right) - \frac{a}{2} \left( \frac{1}{v} + \frac{3}{2v^2} + R_0(v) \right)^2 \right]
\\
&= 1 + \frac{3-a}{2v} + v R_0(v)\\
& \quad - \left( \frac{9a}{8v^3} + \frac{avR_0^2(v)}{2} + \frac{3a}{2v^2} + aR_0(v) + \frac{3a  R_0(v)}{2v} \right)\\
&= 1 + \frac{3-a}{2v} + R_3(v).
\end{align*}
For  $v\ge V_3(\geq V_2\geq V_1\geq 100)$, by $|R_0(v)|\leq \frac{C_0}{v^3}$, we have
\begin{align*}
|R_3(v)|
&= \left| v R_0(v)- \left( \frac{9 a}{8 v^3} + \frac{av R_0^2(v)}{2} + \frac{3 a}{2 v^2} + a  R_0(v) + \frac{3 a R_0(v)}{2 v} \right) \right|
\\
&< v \frac{C_0}{v^3} + \frac{9 a}{8 v^3} +  \frac{av}{2}\cdot \frac{C_0^2}{v^3} + \frac{3 a}{2 v^2} + a\frac{C_0}{v^3} + \frac{3 a}{2 v}\cdot \frac{C_0}{v^3}\\
&= \frac{2 C_0 + a C_0^2 + 3 a}{2 v^2} + \frac{9 a +  8 a C_0}{8 v^3} + \frac{3 a C_0}{2 v^4}
\\
&< \frac{2C_0 + aC_0^2 + 3a}{2v^2} + \frac{9 a + 8 aC_0}{8V_3} \cdot \frac{1}{v^2} + \frac{ 3 aC_0 }{2V^2_3} \cdot \frac{1}{v^2}
\\
&\le \left[ \left( C_0 + \frac{C_0^2}{2}a + \frac{3}{2}a \right) +   \left( \frac{9+8C_0}{8V_0} + \frac{ 3C_0 }{2V^2_0} \right)a \right] \cdot \frac{1}{v^2}
\\
&< \left( C_0 + \frac{C_0^2+5}{2}a \right) \cdot \frac{1}{v^2},
\end{align*}
where we used the inequality $ \frac{9+8C_0}{8V_0} + \frac{ 3C_0 }{2V^2_0}<1$ since $C_0=3$ and $V_0=100$.


By the definition of $R_4(v)$ in Lemma \ref{lemma4}, we know that
\[
R_4(v)= \int_0^{L(v)}\left[\frac{e^{\xi_5}}{2} a^2m^2 + m^2 \left( 1 - am + \frac{e^{\xi_5}}{2} a^2m^2 \right)\left( - \frac{a}{2} + \frac{e^{\xi_6}}{8} {a^2 m^2} \right)\right]dm.
\]
Notice that $\xi_5,\xi_6\in (-1,0)$, $0\leq L(v)<1$ and $0\leq aL(v),aL^2(v)\leq 1$.  Then we get that for \( v\ge V_3 \),
\begin{align*}
|R_4(v)|
&\le  \int_0^{L(v)}\left[a^2m^2 + m^2 (1 + 1 + 1 )( a + 1 )\right] dm
\\
&= \int_0^{L(v)}  (a^2 + 3a + 3) m^2  \, dm
= \frac{a^2 + 3a + 3}{3} L^3(v)\\
&= \frac{a^2 + 3a + 3}{3} \left( \frac{1}{v} + \frac{3}{2v^2} + R_0(v) \right)^3\\
&= \frac{a^2 + 3a + 3}{3} \cdot \frac{1}{v^2} \cdot \left( \frac{1}{v^{\frac{1}{3}}} + \frac{3}{2v^{\frac{4}{3}}} + v^{\frac{2}{3}}\frac{C_0}{v^3}\right)^3\\
&< \frac{a^2 + 3a + 3}{3} \cdot \left( \frac{1}{V_0^{\frac{1}{3}}} + \frac{3}{2V_0^{\frac{4}{3}}} + \frac{C_0}{V_0^{\frac{7}{3}}} \right)^3 \cdot \frac{1}{v^2}
< \left( \frac{a^2}{3} + a + 1 \right) \cdot \frac{1}{v^2}.
\end{align*}
Hence we have
\begin{align}\label{proof-lemma3.4-d}
v \int_0^{L(v)} e^{E_0(1+m)} \, dm
&= 1 + \frac{3-a}{2v}+R_3(v)+R_4(v).
\end{align}

\textbf{Step 2.} In this step, we consider $A_1$.  By the definition of $A_1$ in Lemma \ref{lemma4}, we know that
$$
A_1=v \int_0^{L(v)} e^{E_0(1+m)} \frac{E_1(1+m)}{v}dm.
$$
By (\ref{proof-lemma3.4-b}), we know that $E_0(1+m)<0$. Notice  that $0\leq L(v), aL(v), aL^2(v)\leq 1$ for all $v\geq V_3$. Then for $v\geq V_3$, by (\ref{proof-lemma3.4-c}), we have
\begin{align*}
|A_1|&= \left|v \int_0^{L(v)} e^{E_0(1+m)} \frac{E_1(1+m)}{v}dm\right|
\le  \int_0^{L(v)} | {E_1(1+m)} |dm\\
&=\int_0^{L(v)} \left| (a^2-a)m+ \left( \frac{3}{2}a^2-\frac{a}{2} \right) m^2+a^2m^3+\frac{a^2 m^4}{4}  \right|dm\\
&\le \int_0^{L(v)}\left[(a^2+a)m+ \left( \frac{3a}{2}+\frac{1}{2} \right) m+m+\frac{m}{4}\right]dm\\
&= \left( \frac{a^2}{2}+\frac{5a}{4} + \frac{7}{8} \right) \cdot L^2(v)
= \left( \frac{a^2}{2}+\frac{5a}{4} + \frac{7}{8} \right) \cdot \left( \frac{1}{v} + \frac{3}{2v^2} + R_0(v) \right)^2\\
&= \left( \frac{a^2}{2}+\frac{5a}{4} + \frac{7}{8} \right) \cdot  \left( 1 + \frac{3}{2v} + v R_0(v) \right)^2 \cdot \frac{1}{v^2}
< \left( \frac{a^2}{2}+\frac{5a}{4} + \frac{7}{8} \right) \cdot  \left( 1 + \frac{3}{2V_3} + \frac{C_0}{V_3^2} \right)^2 \cdot \frac{1}{v^2}\\
&\le \left( \frac{a^2}{2}+\frac{5a}{4} + \frac{7}{8} \right) \cdot \left( 1 + \frac{3}{2V_0} + \frac{C_0}{V_0^2} \right)^2 \cdot \frac{1}{v^2}
=\left( 4a^2+10a+7 \right) \cdot \frac{1}{v^2}.
\end{align*}

\textbf{Step 3.} In this step, we consider $A_2$.  By the definition of $A_2$ in Lemma \ref{lemma4}, we know that
$$
A_2=v \int_0^{L(v)} e^{E_0(1+m)} \cdot \frac{R_2(v, 1+m)}{v^2}dm.
$$
By $E_0(1+m)<0$, $|R_2(v,1+m)|\leq C_2$ for all $v\geq V_2$, we get that for all $v\geq V_3$,
\begin{align*}
|A_2|
&= \left| v \int_0^{L(v)} e^{E_0(1+m)} \cdot \frac{R_2(v, 1+m)}{v^2} \, dm \right|  \\
&\le  \int_0^{L(v)} \left| {\frac{R_2(v, 1+m)}{v}} \right| \, dm
\\
&< \left|\frac{C_2}{v} \cdot L(v) \right|
=\left| \frac{C_2}{v} \cdot \left( \frac{1}{v} + \frac{3}{2v^2} + R_0(v) \right) \right|
\\
&< \frac{C_2}{v^2} \cdot \left( 1 + \frac{3}{2V_3} + \frac{C_0}{V_3^2} \right)
\\
&< \frac{C_2}{v^2} \cdot \left( 1 + \frac{3}{2V_0} + \frac{C_0}{V_0^2} \right)
< \frac{2C_2}{v^2}.
\end{align*}

\textbf{Step 4.} By $C_0=3$ and the definition of $C_3$ in Lemma \ref{lemma4}, we get that
\begin{align}\label{proof-lemma3.4-e}
\left( C_0 + \frac{C_0^2+5}{2}a \right) + \left( \frac{a^2}{3} + a + 1 \right) + \left( 4a^2+10a+7 \right) + 2C_2
=\frac{13}{3}y^4 + 18y^2 + 2C_2 + 11=C_3.
\end{align}
By (\ref{proof-lemma3.4-a}), (\ref{proof-lemma3.4-d}), {\bf Step 2, Step 3}, (\ref{proof-lemma3.4-e}), and the definition of $R_G(v)$ in Lemma \ref{lemma4}, we get that for all $v\geq V_3$,
$$
G(v, y) = 1 + \frac{3-a}{2v} + R_G(v), \  |R_G(v)| <  \frac{C_{3}}{v^2}.
$$
The proof is complete.\hfill\fbox

\subsection{Proof of Lemma \ref{lemma5}}

We decompose the proof into seven steps. Recall that $ a=y^2=3 $, \(\phi_1(t) = (1 - 2t)^{-\frac{1}{2}}, V_0=100\), $1\leq x\leq \sqrt{\frac{v}{v-2}}=u(v), m=x-1$ and $L(v)=u(v)-1$.

\textbf{Step 1.} We consider the expansion of $L(v)$. Perform a Taylor's expansion at $t=0$, we have
\[
\phi_1(t) = 1 + t + \frac{3}{2} t^2 + \frac{5}{2} t^3 + \frac{\phi_1^{(4)}(\xi'_1)}{24} t^4,
\]
where, \( \xi'_1 \in (0, t),  \phi_1^{(4)}(t) = 105(1 - 2t)^{-\frac{9}{2}}. \)

 For $0 \le t \le \frac{1}{100}$, we have $|\phi_1^{(4)}(\xi'_1)| \le 105 \cdot {\frac{50}{49}}^{\frac{9}{2}}$ and $v:=\frac{1}{t}\ge 100$. Let $d_0=5$. Then for $v\geq 100$, we have
\begin{align}\label{proof-lem-3.5-s1-a}
L(v)&=\phi_1(\frac{1}{v})-1=\frac{1}{v} + \frac{3}{2v^2} + \frac{5}{2v^3} +\bar{R}_0(v),\\
|\bar{R}_0(v)|
&= \frac{|\phi_1^{(4)}(\xi'_1)|}{24} \cdot \frac{1}{v^4} \le \frac{105 \cdot {\frac{50}{49}}^{\frac{9}{2}}}{24} \cdot \frac{1}{v^4}  < \frac{d_0}{v^4}.\nonumber
\end{align}

\textbf{Step 2.} We consider the expansion of $\ln f(v,x)$. By following (\ref{Proof-Lemma-3.2-a})-(\ref{Proof-Lemma-3.2-c}), and letting $a=3$, we get
\begin{align}\label{proof-lem-3.5-s2-a}
\ln f(v, x)
&=\frac{3}{2} (1-x^2) +\frac{1}{v} \left [\frac{9}{4} (x^4-1)-\frac{3}{2} (x^2-1) \right]\nonumber\\
&  \quad \quad + \frac{1}{v^2} \left [ \frac{9}{4} (x^4-1) + \frac{v+1}{2} \left( \frac{\phi_2 ^{(3)}(\xi'_2)}{6} \cdot \frac{27}{v}-\frac{\phi_2 ^{(3)}(\xi'_3)}{6} \cdot \frac{27 x^6}{v}  \right) \right]\nonumber\\
&:= \bar{E}_0(x) + \frac{1}{v} \bar{E}_1(x) + \frac{1}{v^2} \bar{R}_1(v, x),
\end{align}
where
\begin{align}
&\bar{E}_0(x):= \frac{3}{2} (1-x^2),\quad \bar{E}_1(x) := \frac{9}{4} (x^4-1)-\frac{3}{2} (x^2-1),\label{proof-lem-3.5-s2-b}\\
&\bar{R}_1(v, x):=\frac{9}{4} (x^4-1) + \frac{v+1}{2} \left( \frac{\phi_2 ^{(3)}(\xi'_2)}{6} \cdot \frac{27}{v}-\frac{\phi_2 ^{(3)}(\xi'_3)}{6} \cdot \frac{27 x^6}{v}  \right),\label{proof-lem-3.5-s2-c}
\end{align}
and $\xi'_2 \in (0, t'_1), \xi'_3 \in (0, t'_2)$, \(t'_1 = \frac{3}{v},  t'_2 = \frac{3 x^2}{v}\).

For  $v \ge V'_1:=\max\{100,8a\}=\max\{100,24\}=100$, and  $1 \le x\leq \sqrt{\frac{v}{v-2}} \le \sqrt{2}$, we have that $0<t'_1,t'_2<1$ and thus $\frac{1}{4}<\phi_2^{(3)}(\xi'_2),\phi_2^{(3)}(\xi'_3)<2$. Let $ d_1 = \frac{243}{4} $. Then  by following (\ref{Proof-Lemma-3.2-d}), we have
\begin{align*}
|\bar{R}_1(v, x)|  < 2 \cdot 3^3 + \frac{3}{4} \cdot 3^2 = d_1.
\end{align*}

\textbf{Step 3.} We consider the expansion of $f(v,x)$. By (\ref{proof-lem-3.5-s2-a}), we have
$$
f(v,x)=e^{\bar{E}_0(x)}e^{\frac{1}{v}\bar{E}_1(x)+\frac{1}{v^2}\bar{R}_1(v,x)}.
$$
Define
\begin{align}\label{proof-lemma3.5-s3-a}
    \bar{z}:= \frac{1}{v} \bar{E}_1(x) + \frac{1}{v^2} \bar{R}_1(v, x).
\end{align}
By (\ref{proof-lemma-3.3-bb}) and (\ref{proof-lemma-3.3-bb-1}), we get that for all $v \ge V'_1$,
\begin{align}
    e^{\bar{z}}
    &=1 + \frac{\bar{E}_1(x)}{v}+\frac{1}{v^2}\left(\frac{2\bar{R}_1(v, x) + {\bar{E}_1}^2(x)}{2} + \frac{\bar{R}_1(v, x) \cdot \bar{E}_1(x) }{v} + \frac{{\bar{R}_1}^2(v, x)}{2v^2} + \frac{e^{\xi'_4}}{6} v^2 {\bar{z}}^3\right) \nonumber \\
    &:=1 + \frac{\bar{E}_1(x)}{v}+\frac{\bar{R}_2(v,x)}{v^2},
    \label{proof-lemma3.5-s3-b}
\end{align}
where \( \xi'_4 \) lies between $0$ and \( \bar{z} \), and
\begin{align}\label{proof-lemma3.5-s3-b-1}
\bar{R}_2(v,x):=\frac{2\bar{R}_1(v, x) + {\bar{E}_1}^2(x)}{2} + \frac{\bar{R}_1(v, x) \cdot \bar{E}_1(x) }{v} + \frac{{\bar{R}_1}^2(v, x)}{2v^2} + \frac{e^{\xi'_4}}{6} v^2 {\bar{z}}^3.
\end{align}

For  $v \ge V'_1$, we have  $1 \le x\leq \sqrt{\frac{v}{v-2}} \le \sqrt{2}$.
 By $\bar{E}_1(x) =\frac{9}{4} (x^4-1)-\frac{3}{2} (x^2-1)$, we have
\begin{align}\label{proof-lemma3.5-s3-c}
0 \le \bar{E}_1(x) < \frac{27}{4}.
\end{align}

Notice that $\max \left\{ V'_1, \frac{27}{2}, \sqrt{2d_1} \right\}=\{100,\frac{27}{2}, \sqrt{\frac{171}{4}}\}=100=V_1'$. For $v\geq V'_1$, by (\ref{proof-lemma3.5-s3-c}) and the fact that $|\bar{R}_1(v,x)|< d_1$, we get that
$$
0\leq \frac{\bar{E}_1(x)}{v}\leq \frac{1}{2},\  \frac{|\bar{R}_1(v,x)|}{v^2}\leq \frac{1}{2},
$$
which together with (\ref{proof-lemma3.5-s3-a}) implies that for $v\geq V'_1$, we have $|\bar{z}|\leq 1$, and thus for the $\xi'_4$ in (\ref{proof-lemma3.5-s3-b}), we have
\begin{align}\label{proof-lemma3.5-s3-d}
|\xi'_4|< 1.
\end{align}

Let $ d_2 := \frac{9}{4}d_1 + \frac{25}{32} \cdot 3^4 + \frac{3}{2} \left( d_1^{\frac{2}{3}} \cdot 3^{\frac{4}{3}} + d_1^{\frac{1}{3}} \cdot 3^{\frac{8}{3}} \right) $. For $v\geq V'_1$, by the fact that $|\bar{R}_1(v,x)|\leq d_1$, (\ref{proof-lemma3.5-s3-c}) and (\ref{proof-lemma3.5-s3-d}), we get that
\begin{align*}
|\bar{R}_2(v, x)|
& < \frac{9}{4}d_1 + \frac{25}{32} \cdot 3^4 + \frac{3}{2} \left( d_1^{\frac{2}{3}} \cdot 3^{\frac{4}{3}} + d_1^{\frac{1}{3}} \cdot 3^{\frac{8}{3}} \right)= d_2.
\end{align*}
Hence  for all  \( v\ge V'_1 \),
\[ f(v, x) = e^{\bar{E}_0(x)} \left( 1+\frac{\bar{E}_1(x)}{v}  + \frac{\bar{R}_2(v, x)}{v^2} \right),\ |R'_2(v, x)| < d_2.
\]

\textbf{Step 4. } We consider the expansion of $G(v,y)$. By \textbf{Step 3}, we get that for all $v\geq V'_1$,
\begin{align}\label{proof-lemma3.5-s4-a}
G(v, y) &=v \int_1^{u(v)} f(v, x)dx = v \int_0^{L(v)}f(v, 1+m)dm\nonumber\\
&= v \int_0^{L(v)} e^{\bar{E}_0(1+m)} \left( 1 + \frac{\bar{E}_1(1+m)}{v} + \frac{\bar{R}_2(v, 1+m)}{v^2} \right)dm\nonumber\\
&:= \bar{A}_1 + \bar{A}_2 + \bar{A}_3,
\end{align}
where
\begin{align}
&\bar{A}_1:= v \int_0^{L(v)} e^{\bar{E}_0(1+m)}dm,\quad\quad \bar{A}_2:= v \int_0^{L(v)}
 e^{\bar{E}_0(1+m)}\cdot \frac{\bar{E}_1(1+m)}{v}dm,\label{proof-lemma3.5-s4-ab}\\
&\bar{A}_3:= v \int_0^{L(v)} e^{\bar{E}_0(1+m)}\cdot \frac{\bar{R}_2(v, 1+m)}{v^2}dm.\label{proof-lemma3.5-s4-ac}
\end{align}
In the rest of this step, we consider the first term $\bar{A}_1$.

By (\ref{proof-lem-3.5-s2-b}), we have
\begin{align}\label{proof-lemma3.5-s4-b}
\bar{E}_0(1+m) = -3m-\frac{3m^2}{2},\quad \bar{E}_1(1+m) = 6m+12m^2+9m^3+\frac{9 m^4}{4},
\end{align}
and thus
\begin{align*}
\bar{A}_1&= v \int_0^{L(v)} e^{-3m} \cdot e^{-\frac{3m^2}{2}}dm.
\end{align*}

Notice that $\max\left\{V'_1, 2+\frac{2y^2}{1+2y}, 2+\frac{2y^4}{1+2y^2}\right\}=\max\left\{100, 2+\frac{6}{1+2\sqrt{3}}, 2+\frac{18}{1+6}\right\}=100=V_1'$. For $v\geq V'_1$, one can check that $0\leq L(v)<1$ and $0\leq 3L(v), 3L^2(v)\leq 1$. For $v\geq V'_1$ and $0\leq m\leq L(v)$, similar to Step 1 of Proof of Lemma \ref{lemma4}, we have
\begin{align*}
 e^{-3m}
&= 1 - 3m + \frac{9}{2} m^2 +  \frac{e^{\xi'_5}}{6} \cdot (-3m)^3 =  1 - 3m + \frac{9}{2} m^2 - \frac{9e^{\xi'_5}}{2} m^3, \quad \xi'_5 \in (-3m, 0) \subset (-1,0),
\\
e^{-\frac{3m^2}{2}}
&= 1 - \frac{3m^2}{2} + \frac{e^{\xi'_6}}{2} \left(-\frac{3m^2}{2} \right)^2 =1 - \frac{3m^2}{2} + \frac{9e^{\xi'_6}}{8} {m^4},  \quad \xi'_6 \in (-\frac{3m^2}{2}, 0) \subset (-1,0),
\\
\bar{A}_1
&= v \int_0^{L(v)} \left( 1 - 3m + \frac{9}{2} m^2 - \frac{9e^{\xi'_5}}{2} m^3 \right) \cdot \left( 1 - \frac{3m^2}{2} + \frac{9e^{\xi'_6}}{8} {m^4} \right) \, dm
\\
&= v\int_0^{L(v)}\left[1 - 3m + 3m^2  + \frac{9}{2} m^3 + (1-3m) \cdot \frac{9e^{\xi'_6}}{8} m^4\right.
\\
& \left.\quad\quad \quad\quad \quad + \frac{9}{2} m^2 \left( -\frac{3m^2}{2} + \frac{9e^{\xi'_6}}{8} {m^4} \right) - \frac{9e^{\xi'_5}}{2} m^3  \left( 1 - \frac{3m^2}{2} + \frac{9e^{\xi'_6}}{8} {m^4} \right)dm\right]\\
&= v \cdot \left( L(v) - \frac{3}{2} L^2(v) + L^3(v) \right)
 +\left\{v\int_0^{L(v)}\left[ \frac{9}{2} m^3 + (1-3m) \cdot \frac{9e^{\xi'_6}}{8} m^4\right.\right.\\
&\left.\left.\quad \quad + \frac{9}{2} m^2  \left( -\frac{3m^2}{2} + \frac{9e^{\xi'_6}}{8} {m^4} \right) - \frac{9e^{\xi'_5}}{2} m^3  \left( 1 - \frac{3m^2}{2} + \frac{9e^{\xi'_6}}{8} {m^4} \right)\right]dm
\right\}\\
&:= v \left ( L(v) - \frac{3}{2} L^2(v) + L^3(v) \right) + \bar{R}_3(v).
\end{align*}
Notice that $\xi'_5,\xi'_6\in (-1,0)$ and $0\leq 3m,3m^2\leq 1$ for $0\leq m\leq L(v)$.  Then we get that for all \( v\ge V'_1 \),
\begin{align*}
|\bar{R}_3(v)|
&\le  v\int_0^{L(v)}\left[ \frac{9}{2} m^3 + ( 1 + 1 ) \cdot \frac{3}{8} m^3  + \frac{9}{2} m^2\left( \frac{m}{2} + \frac{m}{8} \right) + \frac{9}{2} m^3 \left( 1 + \frac{1}{2} + \frac{1}{8} \right)\right]dm
\\
&< v \int_0^{L(v)} \left[( 5 + 3 + 3 + 9 ) \cdot  m^3\right]dm\\
&= 5v \cdot  L^4(v) = 5v \left( \frac{1}{v} + \frac{3}{2v^2} +  \frac{5}{2v^3} + \Bar{R}_0(v) \right)^4
\\
&= \frac{5}{v^3} \left(  1 + \frac{3}{2v} +  \frac{5}{2v^2} + v \cdot \Bar{R}_0(v) \right)^4
< \frac{5}{v^3} \cdot \left( 1 + 1 \right)^4
= \frac{80}{v^3}.
\end{align*}

By \textbf{Step 1}, we get
\begin{align*}
&v \left ( L(v) - \frac{3}{2} L^2(v) + L^3(v) \right)
\\
&= v \left[ \left ( \frac{1}{v} + \frac{3}{2v^2} + \frac{5}{2v^3}
+\bar{R}_0(v) \right)   - \frac{3}{2} \left ( \frac{1}{v} + \frac{3}{2v^2} + \frac{5}{2v^3} +\Bar{R}_0(v) \right)^2  + \left ( \frac{1}{v} + \frac{3}{2v^2} + \frac{5}{2v^3} +\Bar{R}_0(v) \right)^3 \right]
\\
&= 1 - \frac{1}{v^2}  + v \cdot \Bar{R}_0(v) -\frac{3v}{2} \left[ \frac{5}{v^4} + \frac{2\Bar{R}_0(v)}{v} + \left( \frac{3}{2v^2} + \frac{5}{2v^3} + \Bar{R}_0(v) \right)^2 \right]
\\
& \quad \quad+ v \left( \frac{3}{2v^2} + \frac{5}{2v^3} + \Bar{R}_0(v) \right) \left[   \left( \frac{1}{v} + \frac{3}{2v^2} + \frac{5}{2v^3}
+\Bar{R}_0(v) \right)^2 + \left( \frac{2}{v^2} + \frac{3}{2v^3} + \frac{5}{2v^4} + \frac{\Bar{R}_0(v)}{v} \right) \right]
\\
&= 1 - \frac{1}{v^2}  + \left\{\frac{v^4 \cdot \Bar{R}_0(v)}{v^3} -\frac{3}{2v^3} \left[ 5 + 2 v^3 \cdot \Bar{R}_0(v) + \left( \frac{3}{2} + \frac{5}{2v} + v^2 \cdot \Bar{R}_0(v) \right)^2 \right]\right.\\
&\left.  \quad \quad+ \frac{1}{v^3} \left( \frac{3}{2} + \frac{5}{2v} + v^2 \cdot \Bar{R}_0(v) \right) \left[   \left( 1 + \frac{3}{2v} + \frac{5}{2v^2} + v \cdot \Bar{R}_0(v) \right)^2 + \left( 2 + \frac{3}{2v} + \frac{5}{2v^2} + v \cdot \Bar{R}_0(v) \right) \right]\right\}\\
&:= 1 - \frac{1}{v^2} + \bar{R}_4(v).
\end{align*}

For all $v\ge V'_1$, by $|\Bar{R}_0(v)|\leq \frac{d_0}{v^4}$, we have
\begin{align*}
|\bar{R}_4(v)|
&<  \frac{1}{v^3} \cdot \left\{ d_0 + \frac{3}{2} \cdot \left[ 5 + \frac{2d_0}{v} + \left( \frac{3}{2} + \frac{1}{2} + \frac{d_0}{v^2} \right)^2 \right]  \right.
\\
& \left. \quad \quad+ \left( \frac{3}{2} + \frac{1}{2} + \frac{d_0}{v^2} \right) \cdot \left[ \left( 1 + \frac{3}{2v} + \frac{5}{2v^2} + \frac{d_0}{v^3} \right)^2 + \left( 2 + \frac{3}{2v} + \frac{5}{2v^2} + \frac{d_0}{v^3} \right) \right]  \right\}
\\
&< \frac{1}{v^3} \cdot
\left\{ d_0 + \frac{3}{2} \cdot \left[ 5+1+\left( 2+1 \right)^2 \right] + \left( 2+1 \right) \cdot \left[ \left( 1+1 \right)^2 + \left( 2+1 \right) \right] \right\}
\\
&= \left( d_0 + \frac{87}{2} \right) \cdot \frac{1}{v^3} .
\end{align*}

Hence for all  \( v\ge V'_1 \), we have
\begin{align}\label{proof-lemma3.5-s4-c}
\bar{A}_1 = 1 - \frac{1}{v^2} + \bar{R}_3(v) + \bar{R}_4(v),\ |\bar{R}_3(v, x)| < \frac{80}{v^3}, \ |\bar{R}_4(v, x)| < \frac{d_0+\frac{87}{2}}{v^3}.
\end{align}

%

\textbf{Step 5.}
In this step, we consider $\bar{A}_2$ in  (\ref{proof-lemma3.5-s4-a}).  By  (\ref{proof-lemma3.5-s4-ab}),   (\ref{proof-lemma3.5-s4-b}) and the Taylor's expansion of the function $e^x$,   we have
\begin{align*}
\bar{A}_2
&= v \int_0^{L(v)} e^{\bar{E}_0(1+m)}  \frac{\bar{E}_1(1+m)}{v}dm\\
&= \int_0^{L(v)} e^{-3m} \cdot e^{-\frac{3m^2}{2}}  \left( 6m+12m^2+9m^3+\frac{9 m^4}{4} \right)dm\\
&= \int_0^{L(v)} \left( 1 - e^{\xi'_7} \cdot 3m \right) \cdot \left( 1 - e^{\xi'_8} \cdot \frac{3m^2}{2} \right) \cdot  \left( 6m+12m^2+9m^3+\frac{9 m^4}{4} \right)dm
\\
&= \int_0^{L(v)}6mdm+\int_0^{L(v)}\left[6m  \left( - e^{\xi'_7} \cdot 3m - e^{\xi'_8} \cdot \frac{3m^2}{2} + e^{ \xi'_7 + \xi'_8 } \cdot \frac{9}{2} m^3 \right)\right.\\
& \left.\quad \quad + \left( 1 - e^{\xi'_7} \cdot 3m \right) \cdot \left( 1 - e^{\xi'_8} \cdot \frac{3m^2}{2} \right) \cdot \left( 12m^2+9m^3+\frac{9 m^4}{4} \right)\right]dm\\
&:= \int_0^{L(v)}6mdm + \bar{R}_5(v),
\end{align*}
where $\xi'_7\in (-3m,0), \xi'_8\in (-\frac{3m^2}{2}, 0)$.

By (\ref{proof-lem-3.5-s1-a}), we have
\begin{align*}
\int_0^{L(v)}  6mdm
&= 3L^2(v)= 3 \cdot \left( \frac{1}{v} + \frac{3}{2v^2} + \frac{5}{2v^3} +\Bar{R}_0(v) \right)^2\\
&= \frac{3}{v^2} + \frac{3}{v^3} \left( \frac{3}{2} + \frac{5}{2v} + v^2 \cdot \Bar{R}_0(v) \right) \cdot \left( 2 + \frac{3}{2v} + \frac{5}{2v^2} + v \cdot \Bar{R}_0(v) \right)\\
&:= \frac{3}{v^2} + \bar{R}_6(v).
\end{align*}

For \( v\ge V'_1 \), we have
\begin{align*}
|\bar{R}_5(v)|
&< \int_0^{L(v)}   6m \left(  3m + \frac{m}{2} + \frac{3m}{2} \right) +  \left( 12m^2+3m^2+\frac{3 m^2}{4} \right)  dm
\\
&< \int_0^{L(v)} 48 m^2 dm = 16 L^3(v)
= 16 \cdot \left( \frac{1}{v} + \frac{3}{2v^2} + \frac{5}{2v^3} +\Bar{R}_0(v) \right)^3
\\
&< \frac{16}{v^3} \cdot \left( 1 + 1  \right)^3 <  \frac{128}{v^3},
\\
|\bar{R}_6(v)| &< \frac{3}{v^3} \cdot \left( \frac{3}{2} + \frac{1}{2} \right) \cdot \left( 2 + 1 \right) = \frac{18}{v^3}.
\end{align*}

Hence for all  \( v\ge V'_1 \), we have
\begin{align}\label{proof-lem3.5-s5-a}
\bar{A}_2 = \frac{3}{v^2} + \bar{R}_5(v) + \bar{R}_6(v),\ |\bar{R}_5(v, x)| < \frac{128}{v^3}, \ |\bar{R}_6(v, x)| < \frac{18}{v^3}.
\end{align}

\textbf{Step 6.} In this step, we consider $A'_3$ in  (\ref{proof-lemma3.5-s4-a}).  By (\ref{proof-lemma3.5-s4-ac}) and (\ref{proof-lemma3.5-s3-b-1}),  we have
\begin{align}\label{proof-lemma3.5-s6-a}
\bar{A_3}
&= v \int_0^{L(v)} e^{\bar{E_0}(1+m)} \cdot \frac{\bar{R_2}(v, 1+m)}{v^2}dm\nonumber\\
&= \int_0^{L(v)} e^{\bar{E_0}(1+m)}\left[ \frac{\bar{R_1}(v, 1+m)}{v}  + \frac{{ \bar{E_1}(1+m)}^2}{2v} \right.\nonumber\\
&\quad\quad \left. + \frac{\bar{R_1}(v, 1+m) \cdot \bar{E_1}(1+m) }{v^2} + \frac{{\bar{R_1}(v, 1+m)}^2}{2 v^3} + \frac{e^{\xi'_4}}{6} v {\bar{z}}^3 \right]dm.
\end{align}

Let $ V'_2 := \max{ \left\{ V'_1, \left[ \frac{9}{ \frac{2}{7^{\frac{1}{3}}}-1 } \right]+1 \right\}  } = 198 $. For $v\geq V_2'$ and  $1\leq x\leq \sqrt{\frac{v}{v-2}}\leq \sqrt{2}$, we have that  $0<t_1'=\frac{3}{v}, t_2'=\frac{3x^2}{v}<\frac{2}{7^{\frac{1}{3}}}-1$, and thus $\frac{7}{4}<\phi_2^{(3)}(\xi'_2),\phi_2^{(3)}(\xi'_3)<2$ for $\xi_2'\in (0,t_1'), \xi_3'\in (0,t_2')$.
In addition, one can check that for $v\geq V_2'$, we have  $0\leq L(v), 3L(v), 3L^2(v)\leq 1$.

Then for $v\geq V'_2$,  by $\bar{E}_0(1+m)\leq 0$, (\ref{proof-lem-3.5-s2-c}), (\ref{proof-lem-3.5-s1-a}) and $|\bar{R}_0|<\frac{5}{v^4}$,   we have
\begin{align*}
& \left| \int_0^{L(v)} e^{\bar{E_0}(1+m)} \cdot \frac{\bar{R_1}(v, 1+m)}{v}dm \right|
\le \int_0^{L(v)}  \left| \frac{\bar{R_1}(v, 1+m)}{v} \right| dm\\
&= \frac{1}{v} \int_0^{L(v)} \left| \frac{9}{4} \left[ (1+m)^4-1 \right] + \frac{v+1}{2} \left[ \frac{\phi_2 ^{(3)}(\xi'_2)}{6} \cdot \frac{27}{v}-\frac{\phi_2 ^{(3)}(\xi'_3)}{6} \cdot \frac{27 \cdot (1+m)^6}{v} \right] \right| \, dm
\\
&\le \frac{9}{4v} \int_0^{L(v)}\left[(4m+6m^2+4m^3+m^4) + \left| \frac{v+1}{v} \left[ \phi_2^{(3)}(\xi'_2)-\phi_2 ^{(3)}(\xi'_3)(1+m)^6 \right] \right|\right]dm\\
&\le \frac{9}{4v} \int_0^{L(v)}\left[(4m+2m+\frac{4}{9}m+\frac{1}{27}m) + \frac{4}{3}\left| \left[ \phi_2^{(3)}(\xi'_2) - \phi_2 ^{(3)}(\xi'_3) \right] - \phi_2 ^{(3)}(\xi'_3) \left[ (1+m)^6 - 1 \right] \right|\right]dm\\
&\le \frac{9}{4v} \int_0^{L(v)}\left[ \frac{175}{27}m + \frac{4}{3}\left| \phi_2^{(3)}(\xi'_2) - \phi_2 ^{(3)}(\xi'_3) \right| + \frac{4}{3} \cdot \left| \phi_2 ^{(3)}(\xi'_3) \left[ (1+m)^6 - 1 \right] \right| \right] dm\\
&< \frac{9}{4v} \int_0^{L(v)} \left[ \frac{175}{27}m + \frac{1}{3} + \frac{8}{3}\left( m^6+6m^5+15m^4+20m^3+15m^2+6m \right) \right] dm
\\
&\le \frac{9}{4v} \int_0^{L(v)}\left[\frac{175}{27}m + \frac{1}{3} + \frac{8}{3} \left( \frac{1}{243}m+\frac{2}{27}m+\frac{5}{9}m+\frac{20}{9}m+5m+6m \right)\right]dm\\
&< \frac{1}{v} \int_0^{L(v)} \left( \frac{3}{4} + 100m \right) dm\\
&= \frac{1}{v} \left( \frac{3}{4}L(v) + 50L^2(v) \right)\\
&= \frac{1}{v} \left[ \frac{3}{4}  \left( \frac{1}{v} + \frac{3}{2v^2} + \frac{5}{2v^3} +\bar{R}_0(v)  \right) + 50 \cdot \left( \frac{1}{v} + \frac{3}{2v^2} + \frac{5}{2v^3} +\bar{R}_0(v)  \right)^2 \right]
\\
&= \frac{3}{4v^2} + \frac{1}{v^3}  \left[ \frac{3}{4} \cdot \left( \frac{3}{2} + \frac{5}{2v} +v^2 \cdot \Bar{R}_0(v)  \right)  + 50 \cdot \left( 1 + \frac{3}{2v} + \frac{5}{2v^2} +v \cdot \bar{R}_0(v) \right)^2 \right]
\\
&< \frac{3}{4v^2} + \frac{1}{v^3}  \left[ \frac{3}{4} \cdot \left( \frac{3}{2} + \frac{1}{2} \right)  + 50 \cdot \left( 1 + 1 \right)^2  \right]
= \frac{3}{4v^2} + \frac{403}{2v^3}.
\end{align*}

For $v\geq V_2'$, by  the fact that $\bar{E}_0(1+m)\leq 0$, we have
\begin{align*}
& \left| \int_0^{L(v)} e^{\bar{E_0}(1+m)} \cdot \left[ \frac{\bar{R}_1(v, 1+m) \cdot \bar{E}_1(1+m) }{v^2} + \frac{e^{\xi_4'}}{6} v {\bar{z}}^3 \right] \,dm \right|
\\
&\le \int_0^{L(v)}\left[ \left| \frac{\bar{R}_1(v, 1+m) \cdot \bar{E}_1(1+m) }{v^2} \right| + \left| \frac{e^{\xi_4'}}{6} v {\bar{z}}^3 \right| \right]dm
\\
&<  \int_0^{L(v)}\left[\frac{d_1 \cdot \frac{27}{4} }{2v^2} + v  \left( \left| \frac{1}{v} \bar{E}_1(1+m) + \frac{1}{v^2} \bar{R}_1(v, 1+m) \right| \right)^3  \right]dm\\
&< \int_0^{L(v)}\left[ \frac{27d_1}{8v^2} + \frac{1}{v^2} \left( \frac{27}{4} + \frac{1}{4} \right)^3  \right]dm
< \int_0^{L(v)} \frac{4d_1 + 343}{v^2} dm\\
&= \frac{4d_1 + 343}{v^2} \cdot L(v)
= \frac{4d_1 + 343}{v^2} \left( \frac{1}{v} + \frac{3}{2v^2} + \frac{5}{2v^3} +\bar{R}_0(v) \right)
\\
&< \frac{4d_1 + 343}{v^3} \cdot \left( 1 + 1 \right)
= \frac{8d_1 + 686}{v^3},
\end{align*}
where we used the inequality $ \frac{3}{2v} + \frac{5}{2v^2} + v \cdot \bar{R}_0(v) \le 1 $ and $ \frac{\bar{R}_1(v, 1+m)}{v} < \frac{1}{4} $ since $|\bar{R}_0(v)|\le \frac{d_0}{v^4}$, $ \left| \bar{R}_1(v, 1+m) \right| \le d_1=\frac{171}{4} $, $v\ge V'_2=198$ and $d_0=5$.

Notice that
$$
\int_0^{L(v)} e^{\bar{E_0}(1+m)} \cdot \left[ \frac{{ \bar{E_1}(1+m)}^2}{2v} + \frac{{\bar{R_1}(v, 1+m)}^2}{2 v^3} \right] \,dm \ge 0.
$$
Then for all  \( v\ge V'_2 \), by (\ref{proof-lemma3.5-s6-a}) and the above discussions, we get
\begin{align}\label{proof-lemma3.5-s6-b}
\bar{A}_3
> -\frac{3}{4v^2} - \frac{403}{2v^3} - \frac{8d_1 + 686}{v^3}
= -\frac{3}{4v^2} - \frac{16d_1+1775}{2v^3}.
\end{align}

\textbf{Step 7.} Define
\begin{align*}
d_3&:= 80 + \left( d_0 + \frac{87}{2} \right) + 128 + 18 + \frac{16d_1+1775}{2} = 1648.
\end{align*}
By (\ref{proof-lemma3.5-s4-a}), (\ref{proof-lemma3.5-s4-c}), (\ref{proof-lem3.5-s5-a}) and (\ref{proof-lemma3.5-s6-b}), we get
\begin{align*}
G \left( v, \sqrt{3} \right)
&= 1 - \frac{1}{v^2} + \bar{R}_3(v) + \bar{R}_4(v) + \frac{3}{v^2} + \bar{R}_5(v) + \bar{R}_6(v) + \bar{A}_3
\\
&> 1 + \frac{2}{v^2}  -  \frac{3}{4v^2} -\frac{d_3}{v^3}
= 1 + \frac{5}{4v^2} -\frac{d_3}{v^3}.
\end{align*}

By (\ref{1.1}), we know
\(
 \max \left\{ V'_2, \ \frac{4}{5}d_3 \right\} = \frac{4}{5}d_3=1318.4=\bar{v}_0 \left( \sqrt{3} \right)
\).
Then for all \( v\ge \bar{v}_0 \left( \sqrt{3} \right) \),
\[
G \left( v, \sqrt{3} \right) > 1 + \frac{ 5v - 4
d_3  }{4v^3} \ge 1.
\]
The proof  is complete. \hfill\fbox

\vskip 0.3cm
\noindent{\large\bf Acknowledgements}

\noindent  This work was supported by the National Natural Science Foundation of China (No. 12171335).


\begin{thebibliography}{100}

\bibitem{BPR23} Bababesi, L., Pratelli, L.,  Rigo, P. On the Chv\'{a}tal-Janson conjecture, Statis. Probab. Lett.,  194 (2023) 109744.


\bibitem{DK22} Dvorak, V., Klein, O. Probability mass of Rademacher sums beyond one standard deviation, SIAM J. Discrete Math., 36 (2022) 2393-2410.



\bibitem{GTH24} Guo, Z.-Y., Tao, Z.-Y., Hu, Z.-C. A study on the negative binomial distribution motivated by Chv\'{a}tal's theroem, Statis. Probab. Lett., 207 (2024) 110037.

\bibitem{HP23} Hollom, L., Portier, J. Tight lower bounds for anti-concentration of Rademacher sums and Tomaszewski's counterpart problem. arXiv:2306.07811v1 (2023).

\bibitem{HLZZ24} Hu, Z.-C., Lu, P., Zhou, Q.-Q., Zhou, X.-W. The infimum values of the probability functions for some infinitely divisible distributions motivated by Chv\'{a}tal's theorem, J. of Math. (Wuhan), 44 (2024) 309-316.

\bibitem{HST24} Hu, Z.-C., Song, R., Tan, Y. On the anti-concentration functions of some familiar families of distributions, Math. Theory Appl., 44 (2024) 1-15.

\bibitem{Jan21} Janson, S. On the probability that a binomial variable is at most its expectation,  Statis. Probab. Lett., 171 (2021) 109020.

\bibitem{KK22} Keller, N., Klein, O. Proof of Tomaszewski's conjecture on randomly signed sums, Adv. Math., 407 (2022) 108558.


\bibitem{LHZ24} Li, C., Hu, Z.-C., Zhou., Q.-Q. A study on the Weibull and Pareto distributions motivated by Chv\'{a}tal's theorem, J. of Math. (Wuhan), 44 (2024) 195-202.

\bibitem{LXH23} Li, F.-B., Xu, K., Hu, Z.-C. A study on the Poisson, geometric and Pascal distributions motivated by Chv\'atal's conjecture, Statis. Probab. Lett.,  200 (2023) 109871.

\bibitem{Rai60} Rainville, E., D. Special Functions[M]. New York: The Macmillan Company, 1960.





\bibitem{Sun21} Sun, P.  Strictly unimodality of the probability that the binomial distribution is more than its expectation, Discrete Appl. Math.,  301 (2021) 1-5.

\bibitem{SHS24a} Sun, P., Hu,  Z.-C., Sun, W. The infimum values of two probability functions for the Gamma distribution, J. Ineq. Appl., 2024 (2024) 5.

\bibitem{SHS24b} Sun, P., Hu,  Z.-C., Sun, W. Variation comparison between infinitely divisible distributions and the normal distribution, Statistical Papers, 65 (2024) 4405-4429.

\bibitem{SHS23} Sun, P., Hu, Z.-C., Sun, W. Variation comparison between the
$F$-distribution and the normal distribution,  {arXiv}: 2305.13615 (2023).

\bibitem{Xie25} Xie, Z.-Y. The exploration of the Rayleigh distribution driven by three
conjectures, Undergraduate thesis, Sichuan University (2025).

\bibitem{ZHS25} Zhang, J., Hu, Z.-C., Sun, W. On the measure concentration of infinitely divisible distributions, Acta Math. Scientia, 45B (2025) 473-492.

\bibitem{Zhao25} Zhao, Y.-A. On the concentration functions of some familiar families of
distributions, Undergraduate thesis, Sichuan University (2025).

\bibitem{ZLH24} Zhou, Q.-Q., Lu, P., Hu, Z.-C. A study on the $F$-distribution motivated by Chv\'{a}tal's theorem, arXiv: 2409.09420v1 (2024).




\end{thebibliography}
\end{document}